\documentclass{article}
  
\pdfoutput=1
\usepackage{graphicx}

\usepackage{mathtools,amsmath,amssymb,amsfonts,amsthm}
\usepackage{mathalfa}

\usepackage{tikz}
\usetikzlibrary{calc}
\usepackage{color}
\usepackage{url}
\usepackage{fullpage}

\usepackage{enumitem}

\usepackage[english]{babel}
\usepackage{sublabel}
\usepackage{tabularx}

\usepackage[boxruled]{algorithm2e}

\usepackage{xcolor}
 \usepackage{caption}
 \usepackage{subcaption}
\usepackage{verbatim}
\usepackage{float}

\usepackage{setspace}

\usepackage{ifthen}

\newtheoremstyle{mythmstyle}{}{}{}{}{\bf}{.}{ }{}

\theoremstyle{plain}
\newtheorem{thm}{Theorem}
\newtheorem{lemma}{Lemma}
\newtheorem{prop}{Proposition}
\newtheorem{cor}{Corollary}
\newtheorem{claim}{Claim}

\theoremstyle{definition}

\newtheorem{defn}{Definition}

\newtheorem{step}{Step}
\newenvironment{stepbis}[1]
  {%
   \addtocounter{step}{-1}%
   \begin{step}}
  {\end{step}}

\theoremstyle{mythmstyle}
\newtheorem{myexample}{Example}
\newtheorem{oss}{Remark}
\newtheorem{assumption}{Assumption}

\newenvironment{example}[1][0]
{ 
  \ifthenelse{\equal{#1}{0}}{
  \myexample
}
{ 
  \renewcommand\themyexample{#1}\myexample
  \addtocounter{myexample}{-1}
}
}
{\endmyexample}

\definecolor{darkgreen}{rgb}{0.0, 0.3, 0.13}
\newcommand{\B}{\mathbb{B}}
\newcommand{\R}{\mathbb{R}}
\newcommand{\N}{\mathbb{N}}
\newcommand{\bS}{\mathbb{S}}
\newcommand{\cD}{\mathcal{D}}
\newcommand{\cC}{\mathcal{C}}
\newcommand{\cK}{\mathcal{K}}
\newcommand{\cN}{\mathcal{N}}
\newcommand{\cR}{\mathcal{R}}
\newcommand{\oell}{\overline{\ell}}
\newcommand{\cQ}{\mathcal{Q}}
\newcommand{\cS}{\mathcal{S}}
\newcommand{\cU}{\mathcal{U}}
\newcommand{\cV}{\mathcal{V}}

\newcommand{\oV}{\overline{V}}
\newcommand{\oz}{\overline{z}}
\newcommand{\virgolette}[1]{``#1''}
\newcommand{\wt}{\widetilde}
\newcommand{\sw}{\textnormal{sw}}

\newcommand{\cPd}{\mathcal{PD}}

\DeclareMathOperator{\Sym}{Sym}
\DeclareMathOperator{\cl}{cl}
\DeclareMathOperator{\Mm}{\mathbf{Mm}}
\DeclareMathOperator{\Mmq}{\mathbf{Mmq\,}}
\DeclareMathOperator{\co}{co}

\DeclareMathOperator{\inn}{int}

\DeclareMathOperator{\rank}{Rank}

\DeclareMathOperator*{\argmax}{arg\,max}
\usepackage{algpseudocode}
  
\begin{document}

\title{Max-Min Lyapunov Functions for Switched Systems \\ and Related Differential Inclusions}
\author{Matteo Della Rossa$^\ast$ \and Aneel Tanwani$^\ast$ \and Luca Zaccarian\thanks{M. Della Rossa, A. Tanwani and L. Zaccarian are with LAAS-CNRS, University of Toulouse (31400), France. L. Zaccarian is also with Dipartimento di Ingegneria Industriale, University of Trento, Italy. Corresponding author: {\tt mdellaro@laas.fr}. \newline This work was supported by the ANR project {\sc ConVan} with grant number ANR-17-CE40-0019-01.}}

\date{}
\maketitle


\begin{abstract}
Starting from a finite family of continuously differentiable positive definite functions, we study conditions under which a function obtained by {\em max-min} combinations is a Lyapunov function, establishing stability for two kinds of nonlinear dynamical systems: a) Differential inclusions where the set-valued right-hand-side comprises the convex hull of a finite number of vector fields, and b) Autonomous switched systems with a state-dependent switching signal.  We investigate generalized notions of directional derivatives for these max-min functions, and use them in deriving stability conditions with various degrees of conservatism, where more conservative conditions are numerically more tractable. The proposed constructions also provide nonconvex Lyapunov functions, which are shown to be useful for systems with state-dependent switching that do not admit a convex Lyapunov function. Several examples are included to illustrate the results.
\end{abstract}

\section{Introduction}
Lyapunov functions play an instrumental role in the stability analysis of dynamical systems;
The textbook \cite{khalil2002nonlinear} and the research monographs \cite{BaccRosi05} and \cite{MaliMaze09} provide an overview of the developments in this field. 
When considering dynamical systems resulting from switching among a finite number of dynamical subsystems described by ordinary differential equations (ODEs) of the form $\dot x = f_i(x)$, $f_i(0) = 0$, $f_i:\R^n \to \R^n$ locally Lipschitz continuous, $i\in \{1,2, \dots, M\}$, different constructions of Lyapunov functions are proposed in the literature to analyze stability of the common equilibrium point: the origin.
Overviews of such methods and related references can be found in \cite{liberzon}, \cite{ShorWirt07}, and \cite{LinAnts09}.

When the evolution of state trajectories results from {\em arbitrary switching } among the individual subsystems, the stability analysis problem is equivalently addressed by considering the differential inclusion (DI), described by
\begin{equation}
\label{eq:sysDI}
\begin{aligned}
\dot x \in \co \bigl \{f_i(x) \, \vert \; i \in \{1, \dots, M\}\bigr\},
\end{aligned}
\end{equation}
where $\co\{S\}$ denotes the convex hull of the set $S$. 
For the \emph{linear} differential inclusion (LDI) case (that is $f_i(x)=A_ix$ for some $A_i \in \R^{n\times n}$), it is shown in \cite{dayawansa}, \cite{molcha} that asymptotic stability is equivalent to the existence of  a common Lyapunov function that is convex, homogeneous of degree 2, and $\mathcal{C}^1(\R^n,\R)$. By addressing a similar question, the paper \cite{MasoBosc06} establishes the existence of a common homogeneous polynomial Lyapunov function for asymptotically stable LDIs.
Various parameterizations can approximate such homogeneous convex functions, such as maximum of quadratic functions and its convex conjugates \cite{goebel2}, \cite{goebel}, which are shown to be universal in \cite{HuBlanchini2010}. Constructions involving functions with convex polyhedral level sets are proposed in \cite{molcha} and in \cite{BlanMian08}. These functions are mostly locally Lipschitz but not continuously differentiable, therefore the notion of set-valued derivatives, studied in \cite[Chapter 2]{clarke3}, \cite{bacciotti99}, is important. Results analyzing nonsmooth Lyapunov functions using such notions of derivatives appear in \cite{ceragioli}, \cite{TeelPral00a}, \cite[Chapter 4]{clarke1}. 
For general differential inclusions, converse Lyapunov theorems are proved in \cite{TeelPral00b}.
Other techniques for constructing common Lyapunov functions come from imposing strong structural assumptions, such as commuting vector fields \cite{NareBala94}, triangular structure or solvability/nil-potency of the Lie-algebra generated by $\{f_i\}_{i = 1}^m$ \cite{LibeHesp99}.
Without any structural conditions, the Lyapunov functions for system~\eqref{eq:sysDI} are, in general, not finitely constructible.

 On the other hand, for discrete-time systems with arbitrary switching, some constructions involving combinatorial methods have recently appeared. Path-complete Lyapunov functions are proposed in \cite{ahmadi} to approximate the joint spectral radius, and it is shown in \cite{angeli} that this class of functions can be written more explicitly in the form of maximum and minimum over a set of smooth functions. Our conference paper \cite{dellarossa18} uses the construction based on max-min of smooth functions to study stability of the continuous-time system \eqref{eq:sysDI} using Clarke's derivative, and this article provides new results in this direction.


In contrast to studying stability uniformly over all possible switching signals as in \eqref{eq:sysDI}, it is also of interest to study dynamical systems driven by a {\em given switching function} $\sigma: \R^n \to \{1, \dots, M\}$, resulting in 
\begin{equation}\label{switching}
\vspace{-0.13cm}
\dot x = f_{\sigma(x)} (x),
\end{equation}
so that the solution set for system~\eqref{switching} is a strict subset of the solution set of system~\eqref{eq:sysDI}.
As already mentioned, existence of a convex Lyapunov function is necessary for asymptotic stability of LDIs. However, it is possible that system \eqref{switching} is asymptotically stable with $\sigma$ fixed, but does not admit a convex Lyapunov function \cite{blanchini}.
It is possible to provide sufficient conditions for a minimum of quadratics (clearly non-convex) to be a Lyapunov function in this context, see \cite{hu} and \cite{xie}.
In general, constructions involving piecewise quadratic functions have been found quite useful \cite{johansson}, and LMI-based formulations have been proposed to compute such functions \cite{DecaBran00}, \cite{PettLenn02}. Beyond piecewise quadratics, sum-of-squares techniques have been used for polynomial Lyapunov functions~\cite{PapaPraj09}.


 
In this article, the problem of interest is to construct a Lyapunov function for systems~\eqref{eq:sysDI} and \eqref{switching} which guarantees asymptotic stability of the origin $\{0\} \subset \R^n$.
We consider the Lyapunov functions obtained by taking the maximum, minimum, or their combination over a finite family of continuously differentiable positive definite functions, see Definition~\ref{def:MaxMin} for details. Such {\em max-min} type of Lyapunov functions were recently proposed in the context of discrete-time switching systems \cite{ahmadi}, \cite{angeli}. For the continuous-time case treated in this paper, studying this class of functions naturally requires certain additional tools from nonsmooth and set-valued analysis, and one such fundamental tool is the generalized directional derivative. In our conference paper \cite{dellarossa18}, we provide stability results based on Clarke's notion of generalized directional derivative for max-min functions. The construction of non-smooth Lyapunov functions for system \eqref{switching} using the Clarke's generalized gradient concept is also presented in  \cite{BaiGru12}. However, this notion turns out to be rather conservative as is seen in several examples (including the one given in Section~\ref{sec:example}). To overcome this conservatism due to Clarke's generalized derivative, we work with the {\em set-valued Lie derivative}, which is formally introduced in Definition~\ref{def:geneder}.
Focusing on this latter notion of generalized directional derivative for the class of max-min Lyapunov functions, the major contributions of this paper are listed as follows:
\begin{itemize}[leftmargin=*]
\item Describe max-min functions and study generalized notions of set-valued derivatives for such functions.
\item Provide stability results for systems \eqref{eq:sysDI} and \eqref{switching} using the set-valued Lie derivative.
\item Obtain stability conditions using matrix inequalities for the case of linear vector fields in \eqref{eq:sysDI} and \eqref{switching}, and Lyapunov functions obtained by max-min of quadratics.
\end{itemize}

The notion of set-valued Lie derivative was introduced in \cite{bacciotti99} for locally Lipschitz $\emph{regular}$ functions. In the context of stability analysis of a differential inclusion, the set-valued Lie derivative was used in \cite{KamRosPar19} to identify
and remove infeasible directions from the differential inclusion. For the {\em max-min} candidate Lyapunov functions studied in this paper, which are not regular in general, we compute set-valued Lie derivatives and use them to derive stability conditions for systems \eqref{eq:sysDI} and \eqref{switching}. The resulting conditions turn out to be less conservative than the ones obtained by using Clarke's derivative in~\cite{dellarossa18}, which are here recovered as a corollary.
When restricting the attention to the linear case $f_i(x) =A_i x$, and max-min functions obtained from quadratic forms, the Lie-derivative conditions require solving nonlinear matrix inequalities.

It should be noted that, since we allow for the minimum operation in the construction, certain elements in our  proposed class of Lyapunov functions are nonconvex. 
In our approach, when we construct a homogeneous of degree 2 nonconvex Lyapunov function for the LDI problem, a convexification of such functions also provides a Lyapunov function \cite[Proposition 2.2]{goebel}. In fact, the sublevel sets of max-min functions approximate the convex sublevel sets of a homogeneous of degree 2 convex Lyapunov function (which is known to exist) with nonconvex sets obtained via intersections and unions of ellipsoids.

 When addressing  system \eqref{switching}, our approach provides a more general class of nonconvex and nondifferentiable Lyapunov functions obtained via max-min operations. To describe the solutions of switched  systems, we adopt Filippov regularizations \cite{filippov1988differential}, and establish stability conditions for the resulting system. Considering such regularized differential inclusions for the switched systems also allows considering {\em sliding motions} along the switching surfaces.  In this setting, our adopted notion of  set-valued Lie derivative turns out to be crucial and has an interesting geometrical interpretation in terms of the tangent subspace to the switching surface.


The paper is organized as follows: In Section \ref{sec:example} we provide an example of a two-dimensional switched system that does not admit a convex Lyapunov function, but a max-min Lyapunov function can be found.
In Section \ref{sec:preliminaries} we introduce generalized notions of derivatives for Lipschitz continuous functions, while in Section \ref{sec:maxmin} the class of \emph{max-min}\ functions is presented and we show our main stability results in the setting of differential inclusions. In Section \ref{sec:swSys} we apply our results to switched systems, written as a differential inclusion using Filippov regularizations, and we study asymptotic stability along with an instructive example. In Section \ref{sec:linear}, we analyze deeply the case of \emph{linear} switched systems and propose an algorithmic procedure to construct max-min Lyapunov functions, followed by some concluding remarks in Section~\ref{sec:conc}.

\section{A Motivating Example}\label{sec:example}

 We consider a switched system for which there does not exist any convex Lyapunov function. However, this system is asymptotically stable and our results will allow constructing a Lyapunov function $V$ defined as
 \begin{equation}\label{eq:exShieldLyap}
 V(x) := \max\left\{\min\{x^\top P_1 x, x^\top P_2 x\}, x^\top P_3 x \right\},
\end{equation}
for some positive definite matrices $P_i \in \R^{2 \times 2}$, $i = 1,2,3$.
This example was introduced in \cite{dellarossa18} where we did not include the proof of Proposition \ref{prop:notexistconv}, given below.

\begin{example}\label{shieldexample}
Consider a linear switched system as in \eqref{switching},  with three subsystems and a state-dependent switching rule $x \mapsto \sigma(x) \in \{1,2,3\}$, namely
\begin{equation}\label{eq:exShieldSys}
\dot x =A_{\sigma(x)}x
\end{equation}
where 
$
(A_1, A_2,A_3)= \left(\begin{bsmallmatrix}
    -0.1 & 1 \\
    -5 & -0.1
  \end{bsmallmatrix},
\begin{bsmallmatrix}
    -0.1 & 5 \\
    -1 & -0.1
  \end{bsmallmatrix},
\begin{bsmallmatrix}
   1.9 & 3 \\
    -3 & -2.1
    \end{bsmallmatrix}\right).
$
To define the switching signal $\sigma$, introduce matrices
\begin{equation*}
(Q_1, Q_2,Q_3)\!:= \! \left(\begin{bsmallmatrix}
   -(1+\sqrt{2})  & -\frac{2+\sqrt{2}}{2} \\
-\frac{2+\sqrt{2}}{2} & -1
  \end{bsmallmatrix},
\begin{bsmallmatrix}
   \frac{-1}{(1+\sqrt{2}) } & -\frac{\sqrt{2}}{2} \\
-\frac{\sqrt{2}}{2} & -1
  \end{bsmallmatrix},
\begin{bsmallmatrix}
    1 & \sqrt{2}\\
\sqrt{2} & 1 
\end{bsmallmatrix}\right)
\end{equation*} 
and the switching signal
\begin{equation}\label{eq:signal}
\sigma(x):=\begin{cases} 
1, \hskip0.5cm \text{if }\; x \in\mathcal{S}_1:=\{x^\top Q_1x > 0\} \cup \cS_{13}, \\
2,\hskip0.5cm \text{if } \;x \in\mathcal{S}_2:=\{x^\top Q_2x > 0\} \cup \cS_{21},\\
3, \hskip0.5cm \text{if }\;x\in\mathcal{S}_3:=\{ x^\top Q_3x > 0\} \cup \cS_{32},
\end{cases}
\end{equation}
 where the subspaces $\cS_{ij}$, $i \neq j$ are defined as
 $
 \cS_{ij} := \{x \in \R^2 \, \vert \, x^\top Q_i x = x^\top Q_j x\}
 $
, namely
 \begin{equation*}
 \begin{aligned} 
 \cS_{13}&:= \left\{x \in \mathbb{R}^2 \, \vert \, x_2=-(1+\sqrt{2})x_1\right\},\\
 \cS_{21}&:=\left\{\;x \in \mathbb{R}^2 \, \vert \, x_2=-x_1\right\},  \\
\cS_{32}&:= \left\{x \in \mathbb{R}^2 \, \vert \, x_2=-\frac{1}{1+\sqrt{2}}x_1\right \}.\\
\end{aligned}
\end{equation*}
We note that in \eqref{eq:signal}, we have $\mathcal{S}_1 \cup \mathcal{S}_2\cup \mathcal{S}_3= \R^2$ and that the only point of intersection among the three sets is the origin.
\begin{prop}\label{prop:notexistconv}
There does not exist a convex Lyapunov function for system~\eqref{eq:exShieldSys}.
\end{prop}
\begin{proof}
Given a set $\cR_0\subset \R^2$ and a time $T>0$, let $\cC(T;\cR_0)$ be the set of reachable points of solutions of system (\ref{eq:exShieldSys}) after time $T$, starting in $\cR_0$, that is,
\[
\cC(T;\cR_0) := \left\{x(t)\in \R^n \vert \, x \text{ solves \eqref{eq:exShieldSys}},\; x(0) \in\cR_0,\; t \geq T\right\}.
\]
Following \cite[Lemma 2.1]{blanchini}, if we show that there exists a compact set $\cR_0\neq \{0\}$ and a $T>0$ such that $\cR_0 \subset \co \{\cC(T;\cR_0)\}$ (where  $\co\{S\}$ is the convex hull of $S$) then the system does not admit a convex Lyapunov function.
\begin{figure}
\begin{center}
\includegraphics[scale=0.8]{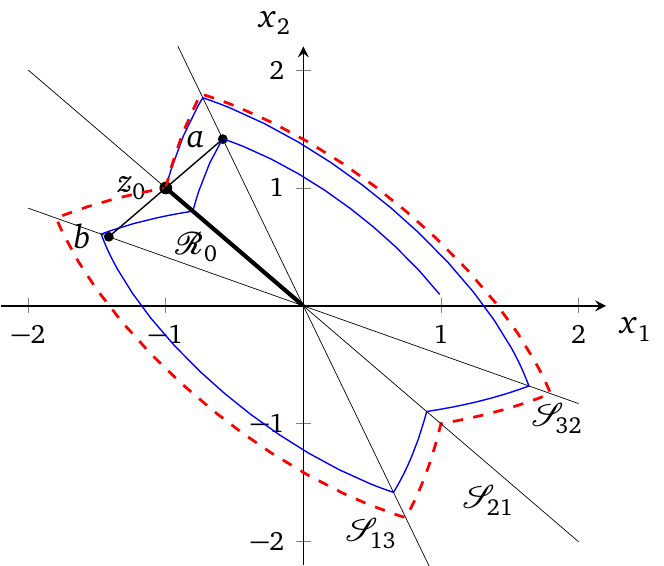}
\caption{The solid blue line shows a trajectory of system (\ref{eq:exShieldSys}) starting at $z_0$ and moving in the clockwise direction. The  red dashed line indicates a level set of the max-min Lyapunov function \eqref{eq:exShieldLyap}. The solid black line indicates  the set $\mathcal{R}_0$ used in the analysis.}
\label{flower2plot}
\end{center}
\end{figure}
Toward this end, we choose $z_0:=[-1~1]^\top\in \mathcal{S}_{21}$, and the compact set $\cR_0 := \{\alpha z_0 : \alpha \in [0,1]\} \subset \mathcal{S}_{21}$, i.e. the line segment connecting $0$ and $z_0$.
We compute
\begin{equation*}
\begin{aligned}
\allowbreak
& e^{A_1t}=e^{-\frac{t}{10}}\begin{bsmallmatrix}
    \cos(\sqrt{5}t) & \frac{\sqrt{5}}{5} \sin(\sqrt{5}t) \\
   - \sqrt{5} \sin(\sqrt{5}t)  &   \cos(\sqrt{5}t)
  \end{bsmallmatrix},\\
&  e^{A_2t}=e^{-\frac{t}{10}}\begin{bsmallmatrix}
    \cos(\sqrt{5}t) &  \sqrt{5}\sin(\sqrt{5}t) \\
   -  \frac{\sqrt{5}}{5}\sin(\sqrt{5}t)  &   \cos(\sqrt{5}t)
  \end{bsmallmatrix},\\
 & e^{A_3t}=e^{-\frac{t}{10}}\begin{bsmallmatrix}
    \frac{2}{\sqrt{5}}\sin(\sqrt{5}t)+\cos(\sqrt{5}t) & \frac{3}{\sqrt{5}}\sin(\sqrt{5}t) \\
   - \frac{3}{\sqrt{5}}\sin(\sqrt{5}t)  &   \cos(\sqrt{5}t)-\frac{2}{\sqrt{5}}\sin(\sqrt{5}t)
  \end{bsmallmatrix},\\
\end{aligned}
\end{equation*}
which allows us to write analytically the solution of the system starting from any given initial condition.
We let $t_1 > 0$ be the smallest time such that
$z_1 := e^{A_1t_1} z_0 \in \cS_{13}$,
and $t_2 > t_1$ be the smallest time such that 
$z_2 := e^{A_3(t_2-t_1)} z_1 \in \cS_{32}$.
We finally choose $t_3> t_2$ as the smallest time such that $z_3:= e^{A_2(t_3-t_2)} z_2 \in \cS_{21}$. It turns out that $|z_3|=1.2671$.
Thus, the half turn, starting with $z_0 \in \cS_{21}$ and reaching $z_3 \in \cS_{21}$, decreases the norm of the state by a factor of $\beta:=\frac{|z_3|}{|z_0|}=0.8961$. Due to the central symmetry of the dynamics (that is, if $x$ is a solution, then $-x$ is also a solution) the solution will reach the set $\mathcal{R}_0$ at the point $\beta^2 [-1 \;1]^\top$ at time $\wt t_3 =2 t_3$. Hence, the set $\mathcal{R}_1:=\{\alpha z_0: \alpha \in [0,\beta^2]\}$ is (strictly) contained in the set $\cC(t_1;\mathcal{R}_0)$. To show that $\cR_0 \subset \co \{\cC(t_1;\cR_0)\}$, it thus remains to check that
\begin{equation}\label{eq:tempInc}
\{\alpha z_0: \alpha \in [\beta^2,1]\} \subset \co\{\cC(t_1;\mathcal{R}_0)\}.
\end{equation}
Property \eqref{eq:tempInc} is graphically illustrated in Figure~\ref{flower2plot} and is proven by the fact that points $a=[\sqrt{2}-2 \;\sqrt{2}]^\top \in \cS_{13}$ and $b=[-\sqrt{2} \; 2-\sqrt{2}]^\top \in \cS_{32}$ satisfy $|a|<|z_1|$ and $|b|<|e^{A_3(t_2-t_1)}e^{A_1t_1}z_3|$, and thus $a,b\in \cC(t_1;\mathcal{R}_0)$, and $z_0=\frac{1}{2}a+\frac{1}{2}b \in \co\{\cC(t_1;\mathcal{R}_0)\}$. Having already shown that $0 \in \co\{\cC(t_1;\mathcal{R}_0)\}$, property \eqref{eq:tempInc} indeed holds.
\end{proof}
In Section~\ref{sec:swSys}, we will study conditions that lead to the construction of a Lyapunov function for state-dependent switched systems. In particular, for the aforementioned example, we will find matrices $P_i > 0$, $i = 1,2,3$  to show that the function $V$ in \eqref{eq:exShieldLyap} is a Lyapunov function.
\end{example}
\section{Generalized Gradients and Directional Derivatives}\label{sec:preliminaries}
The function in \eqref{eq:exShieldLyap} is nonsmooth and requires generalized notions of gradient and directional derivatives, recalled here from \cite[Chapter 2]{clarke3}, \cite{ceragioli}.

Let $F: \R^n \rightrightarrows \R^n$ be an upper semicontinuous\footnote{ A set-valued map $F:\R^n \rightrightarrows \R^n$ is said to be \emph{upper semicontinuous at }$x$ if, for every $\varepsilon>0$ there exists a $\delta >0$ such that if $y \in \B(x,\delta)$ then $F(y) \subset F(x)+\B(0, \varepsilon)$. It is said to be \emph{upper semicontinuous} if it is upper semicontinuous at every $x \in \R^n$. For comparisons with a related notion of {\em outer semicontinuity}, see \cite[Lemma~5.15]{GoebSanf12}.} map with nonempty, compact, convex values, and consider the differential inclusion (resembling dynamics \eqref{eq:sysDI} and \eqref{switching}),
\begin{equation}\label{eq:diffinc}
\dot x \in F(x), \quad x(0) = x_0\in \R^n.
\end{equation}
We recall that a solution of $\eqref{eq:diffinc}$ on an interval $[0,T) \subset \R$ is a function $x:[0,T) \to \R^n$ such that $x(\cdot)$ is absolutely continuous, $x(0) = x_0$, and $\dot x(t) \in F(x(t))$ for almost all $t\in [0,T)$.
In the case $T=+\infty$ the solution is said to be \emph{complete}. 
The origin of \eqref{eq:diffinc} is {\em asymptotically stable} (AS) if it is {\em Lyapunov stable} (for each $\varepsilon >0\;\exists\,  \delta(\varepsilon)>0$ such that all solutions satisfy $|x(0)| < \delta(\varepsilon) \, \Rightarrow  \,|x(t)| < \varepsilon$, $\forall t >0$) and {\em attractive} (there exists  $M>0$ such that solutions satisfying $|x(0)|<M \Rightarrow \lim_{t \to \infty} |x(t)| =0$.)
If attractivity is global (it holds for every $M>0$), then we say that the origin is \textit{globally asymptotically stable} (GAS). We are only concerned with stability of the origin in this article, and use the statement that a system is (G-)AS to refer to the stability of the origin for the corresponding system.
Given an open and connected set $\cD\subset \R^n$ such that $0\in \cD$, we say that a locally Lipschitz function  $V:\cD\to \R$ is a \emph{Lyapunov function} for~\eqref{eq:diffinc} if  there exist class $\cK$ functions\footnote{A function $\alpha:\R_{\geq 0}\to \R$ is \emph{positive definite} ($\alpha\in\cPd$) if it is continuous, $\alpha(0)=0$, and $\alpha(s)>0$ if $ s\neq 0$. A function $\alpha:\R_{\geq 0}\to \R_{\geq 0}$ is \emph{class $\cK$} ($\alpha \in \cK$) if it is continuous, $\alpha(0)=0$, and strictly increasing. It is said to be \emph{class $\cK_\infty$} if it is class $\cK$ and unbounded.}
 $\underline{\chi},\overline{\chi}$, and a positive definite $\gamma\in \cPd$
 such that
\[
\underline{\chi}(|x|)\leq V(x)\leq \overline{\chi}(|x|),\;\;\forall x\in \cD,
\] 
and there exists a $\delta >0$ such that given any solution $x:[0,T)\to \cD$ of \eqref{eq:diffinc} with $|x(0)|<\delta$, we have
\[
\frac{d}{dt}V(x(t))\leq -\gamma(|x(t)|), \;\;\text{for almost every } t\in [0,T).
\]
The existence of a Lyapunov function implies the asymptotic stability of system \eqref{eq:diffinc}.
If moreover $\cD=\R^n$, $\underline{\chi},\overline{\chi}\in \cK_\infty$ and $\delta$ can be arbitrarily large, the existence of such a $V$ implies global asymptotic stability of \eqref{eq:diffinc}, see \cite[Chapter 4]{khalil2002nonlinear}.
Given an open set $\cD\subset \R^n$ and a locally Lipschitz function $V:\cD\to \R$ we first consider the Clarke's generalized gradient $x \mapsto \partial V(x)$ \cite[Chapter 2]{clarke3}, which, due to the equivalence in \cite[Theorem 2.5.1, page 63]{clarke3} can be defined as
\begin{equation}\label{eq:limClark}
\partial V(x):= \co \left\{v\in \R^n\,{\Big \vert}\begin{aligned} \,  &\exists\, x_k \to x, \;x_k \notin \cN_V, \text{ s.t.} \\ &v=\lim_{k \to \infty} \nabla V(x_k)\end{aligned} \right\},
\end{equation}
 where $\cN_V \subset \R^n$ is the set of  measure zero where $\nabla V$ is not defined. \cite[Theorem 2.5.1, page 63]{clarke3} proves the existence of at least one sequence $x_k$ as considered in~\eqref{eq:limClark}, namely $\partial V(x)\neq \emptyset$, for all $x \in \cD$. 
  Moreover, the following property of locally Lipschitz functions will be used in what follows.
  \begin{defn}\label{def:Regular}
Given an open set $\cD\subset \R^n$, a locally Lipschitz function $V:\cD\to \R$ is \emph{regular} at $x\in \cD$ if, for every $v\in \R^n$, the directional derivative
$
V'(x;v):=\lim_{h\to 0^+}\frac{V(x+hv)-V(x)}{h}
$ exists and the equality
\begin{equation}\label{eq:regularity}
V'(x;v)=\max\left\{w^\top v\;\vert\;w\in \partial V(x)\right\},\;\;\;\forall v\in \R^n,
\end{equation}
holds. $V$ is called \emph{regular} if it is regular at each $x\in \cD$.
  \end{defn}
Definition~\ref{def:Regular} is in fact a characterization of regularity for locally Lipschitz functions, which follows from~\cite[Proposition~2.1.2]{clarke3}. For an alternative definition we refer to~\cite[Definition 2.3.4]{clarke3}. The right-hand side of~\eqref{eq:regularity} is also called the \emph{Clarke's generalized directional derivative} of $V$ at $x$ along $v$ (denoted by $V^0(x,v)$ and defined in~\cite[Section 2.1]{clarke3}). The results of this paper could be equivalently stated by referring to Clarke's generalized directional derivatives instead of the Clarke's generalized gradient in~\eqref{eq:limClark}, but we believe that the gradient is a more familiar concept in the control community.

We now introduce two different notions of the generalized directional derivative with respect to differential inclusions~\eqref{eq:diffinc}.
 We will show that the first one, the more \virgolette{natural} one, leads to more conservative stability results than the second one. In particular the second one is needed for proving GAS of the motivating example introduced in Section~\ref{sec:example}.

\begin{defn}[\cite{bacciotti99, cortes}]\label{def:geneder}
Consider the differential inclusion \eqref{eq:diffinc}; given an open set $\cD\subset \R^n$ and a locally Lipschitz function $V:\cD \to \R$. Given $x\in \cD$, the \emph{Clarke's generalized derivative} of $V$ with respect to $F$ is defined as
\begin{equation}\label{eq:oldgenderiv}
\dot V_F(x) := \{p^\top f \; | \; p \in \partial V(x),\, f \in F(x) \, \}.
\end{equation}
Additionally, we define the \textit{set-valued Lie derivative} of $V$ with respect to $F$ as
\begin{equation}\label{eq:genderiv}
\dot{\overline{V}}_F(x):=\{ a \in \R \;|\; \exists f \in F(x): p^\top f=a, \, \forall p \in \partial V(x) \}.
\end{equation}
\end{defn}
In the case where $V$ is continuously differentiable at~$x$, one has $\partial V(x)=\{\nabla V(x)\}$ and $\dot{\overline{V}}_F(x)=\dot V_F(x)=\{\nabla V(x)^\top f \, \vert \, f \in F(x) \}$.
Moreover, it is clear that
\begin{equation}\label{eq:setdiffinclusion}
\dot{\overline{V}}_F(x) \subset \dot V_F(x).
\end{equation}
In fact, given $a \in \dot{\overline{V}}_F(x)$, there exists $f \in F(x)$ such that $a=p^\top f$, for all $p \in \partial V(x)$ and thus in particular $a \in \dot V_F(x)$.
Intuitively, it means that when defining  $\dot{\overline{V}}_F(x)$ we do not consider \emph{every} possible scalar product between vectors of $\partial V(x)$ and $F(x)$, rather we only consider  directions $f \in F(x)$ that are  \virgolette{meaningful} in the sense of possible flowing directions of solutions.
Recalling that the Euclidean scalar product $\langle \cdot, \cdot \rangle$ is bilinear in its arguments and, for each $v \in \R^n$,  $\langle v,\cdot \rangle$ is continuous, it can be shown that, for each fixed $x \in \cD$,  $\dot{\overline{V}}_F(x)$ and $\dot V_F(x)$ are compact intervals, possibly empty.
Concluding this section, we illustrate the differences between the different notions of set-valued derivatives in the following example.
\begin{example}\label{ex:OneDimExample}
Consider the function $V:\R \to \R$ defined as $V(x)=|x|$, which is differentiable everywhere except at $0$ and Lipschitz continuous. 
From~\eqref{eq:limClark}, Clarke's generalized gradient at 0 is
$\partial V(0)=[-1,1]$.
Now let us suppose that a set-valued map $F:\R \rightrightarrows \R$ is given such that  $F(0):=[f_1, f_2] \subset \R$. Using~\eqref{eq:oldgenderiv}, we compute
\begin{equation*}
\begin{aligned}
\dot V_F(0)&=\{pf \;\vert \; p \in [-1,1],\, f \in [f_1,f_2]\}\\
           &=\left [-\max\{|f_1|, |f_2|\}, \max\{|f_1|,|f_2|\} \right].
\end{aligned}
\end{equation*}
On the other hand, using~\eqref{eq:genderiv} and noting that  $p_1f = p_2f$ for each $p_1,p_2 \in [-1,1]$ if and only if $f=0$,
we get
\begin{equation*}
\dot{\overline{V}}_F(0)=\begin{cases}
\{0\} \; &\text{if} \;0 \in [f_1,f_2],\\
\emptyset \;\; &\text{if}\; 0 \notin [f_1,f_2].
\end{cases}
\end{equation*}
It is easily verified that $\dot{\overline{V}}_F(0)$ is a subset of $\dot V_F(0)$. 
\end{example}
\section{Stability Using Max-Min Functions}\label{sec:maxmin}
In this section, we use the generalized derivatives to study a particular class of locally Lipschitz Lyapunov functions establishing sufficient stability conditions for system~\eqref{eq:sysDI}.

\subsection{Max-Min Functions}
The following definition was introduced by \cite{angeli} in the context of path-complete Lyapunov functions for discrete time switching systems. 
\begin{defn}\label{def:MaxMin}
Consider an open and connected set $\cD\subset \R^n$. Given $K$ \emph{base} functions $V_1, \dots, V_K \in \mathcal{C}^1(\cD,\mathbb{R})$, a {\em max-min function} $V_{\Mm} : \cD \to \R$ is either defined as
 \begin{subequations}\label{eq:MaxMin/MinMax}
\begin{align}\label{eq:Mmm}
V_{\Mm}(x)&:=\max_{j \in \{1, \dots, J\}} \left \{\min_{k \in S_j} \{V_k(x)\} \right \},
\end{align}
for some $J\geq1$ and nonempty sets $S_1, \dots, S_J\subset \{1, \dots, K\}$, or
\begin{align}
V_{\Mm}(x) &=\min_{j \in \{1, \dots, J^\star\}} \left \{\max_{k \in S^\star_j} \{V_k(x)\} \right\},\label{eq:MinMax}
\end{align}
\end{subequations}
for some $J^\star \geq1$ and nonempty sets $S^\star_1, \dots, S^\star_{J^\star} \subset \{1, \dots, K\}$.
\end{defn} 
The following proposition states the equivalence between \eqref{eq:Mmm} and \eqref{eq:MinMax}, which is obtained by applying the distrubutive property of the $\max$ $\min$ operators. For a formal proof we refer to~\cite{Ovc10} and references therein. In the sequel, all our derivations apply to both equivalent expressions \eqref{eq:Mmm} and~\eqref{eq:MinMax} but for definiteness, we use the notation adopted in \eqref{eq:Mmm}.

\begin{prop}\label{prop:Equivalence}
Given $J\geq 1$ (resp. $J^\star\geq 1$), and $S_1,\dots S_J$ (resp. $S^\star_1,\dots, S^\star_{J^\star}$) nonempty subsets of $\{1,\dots, K\}$, there exists $J^\star\geq 1$ (resp. $J$)  and nonempty subsets $S^\star_1,\dots,S^\star_{J^\star}$ (resp. $S_1,\dots S_J$) of  $\{1,\dots, K\}$ such that expressions \eqref{eq:Mmm} and \eqref{eq:MinMax} coincide, for all $x\in\cD$, and  for any  $V_1,\dots ,V_K\in \cC^1(\cD,\R^n)$.
\end{prop}

We denote by $\Mm(V_1, \dots, V_K)$ the set of all the possible max-min functions obtained from $K$ base functions $V_1, \dots, V_K$. 
Given  $V \in \Mm (V_1, \dots, V_K)$, it is noted that at each point $x\in \cD$ where a strict ordering holds between the values of the base functions, that is, $V_{\ell_1}(x) <V_{\ell_2}(x) < \dots <V_{\ell_K}(x)$,
the function value $V(x)$ coincides with $V_{\wt \ell}(x)$, for some $\wt \ell \in \{1, \dots, K\}$.
At points where two or more base functions are equal, the function $V$ may switch between different base functions. For every $\ell\in\{1,\dots K\}$, we may define the set where the function $V_\ell$ is active, more precisely
\begin{equation}\label{eq:actset}
C_\ell:=\{x\in \cD\;\vert\;V(x)=V_\ell(x)\},
\end{equation}
which are closed by continuity of $V, V_1, \dots, V_K$.
We can  associate a mapping with every $V \in \Mm(V_1, \dots, V_K)$. This map is useful to characterize the generalized derivatives introduced in Definition~\ref{def:geneder}.
\begin{defn}[Essentially-active index map]\label{def:alpha}
Given a function $V\in \Mm(V_1, \dots, V_K)$,  the corresponding {\em essentially-active index map} $\alpha_V: \cD \rightrightarrows \{1, \dots, K\}$ is defined as
\begin{equation}\label{eq:alpha}
\alpha_V(x):=\big\{\ell\in\{1,\dots, K\}\;\vert\; x\in \cl(\inn(C_\ell)) \big \},
\end{equation}
where $\cl(C)$ and $\inn(C)$ represent the closure and the interior of a set $C\subset\R^n$, respectively. Indexes $\ell\in \alpha_V(x)$ are called \emph{essentially-active indexes} of $V$ at $x$.
\end{defn}
It will be shown in~Lemma~\ref{lemma:prelimnLemma} that $\alpha_V(x)$ is nonempty, for every $x\in \cD$. Here, instead, we highlight that
\begin{equation}\label{eq:alfainclusion}
\alpha_V(x) \subset\{ \ell \in \{1, \dots, K\} \, \vert \, V(x)=V_{\ell}(x)\},\;\;\forall x\in \cD.
\end{equation}
The set appearing in the right-hand side of inclusion~\eqref{eq:alfainclusion} is called~\emph{active index set} in the context of piecewise $\cC^1$ functions, for example in~\cite{PanRal96} and~\cite[Chapter 4]{Sch12}.
To obtain the inclusion~\eqref{eq:alfainclusion}, consider any $\ell \in \alpha_V(x)$, then from Definition~\ref{def:alpha} and $\cD$ being open, there is a sequence $x_k \to x$ such that $x_k\in \inn(C_\ell)$, $\forall\;k\in \N$. By continuity of $V$ and $V_\ell$, we have
$
V(x)= \lim_{k \to\infty}V(x_k)= \lim_{k \to \infty} V_{\ell}(x_k)= V_{\ell}(x)$.


 We emphasize that, in general, the inclusion in \eqref{eq:alfainclusion} is strict and equality does not necessarily hold. 

 Moreover, given $V\in \Mm(V_1, \dots, V_K)$, the map $\alpha_V:\R^n\rightrightarrows\{1,\dots, K\}$ contains all the necessary information to locally describe the function $V$, as formalized in the following result. 
\begin{lemma}\label{lemma:prelimnLemma}
Consider $V\in \Mm(V_1,\dots, V_K)$. For each $x\in \cD$ the set $\alpha_V(x)$ is non empty and there exists a neighborhood $\cU$ of $x$ such that 
\begin{equation}\label{eq:lemmaImplication}
(z\in \cU)\Rightarrow(\exists\,\ell_z\in \alpha_V(x)\text{ such that }V(z)=V_{\ell_z}(z)).
\end{equation}  
\end{lemma}
The proof of Lemma~\ref{lemma:prelimnLemma} is given in Section~\ref{sec:propGrad}.
\subsection{Gradients and Stability Conditions}
The following statement draws connections between Clarke's generalized gradient $\partial V$ and the  set-valued Lie derivative $\dot{\overline{V}}_F$ in \eqref{eq:genderiv} for a generic $V \in \Mm(V_1, \dots, V_K)$, using the mapping $\alpha_V$. 
\begin{prop}\label{lm:gengrad}
Given $V \in \Mm(V_1, \dots, V_K)$ and $x\in \cD$, the following equality holds
\begin{equation}\label{eq:gradConvex}
\partial V(x)=\co\{ \nabla V_\ell(x) \, \vert \, \ell\in \alpha_V(x)\}.
\end{equation}
In particular, given  $F :\R^n \rightrightarrows \R^n$, the Lie derivative in~\eqref{eq:genderiv} reads
\begin{equation}\label{eq:Mmdirder}
\dot{\overline{V}}_F(x) = \{ a \in \R\, \vert \, \exists f \in F(x) : a =\nabla V_\ell(x)^\top f,\, \forall \ell \in \alpha_V(x)\}.
\end{equation}
\end{prop}
\begin{proof}
Max-min functions are in particular piecewise $\cC^1$ functions, as defined in~\cite[Chapter~4]{Sch12}. Then, equation~\eqref{eq:gradConvex} is proved following the arguments presented in~\cite[Lemma 2]{PanRal96} or in~\cite[Proposition 4.3.1]{Sch12}. Combining~\eqref{eq:gradConvex} with the definition of $\dot{\overline{V}}_F$ given in \eqref{eq:genderiv}, we obtain~\eqref{eq:Mmdirder}.
\end{proof}

We now propose a sufficient condition for asymptotic stability of system \eqref{eq:diffinc} in terms of $\dot{\overline{V}}_F$ given in \eqref{eq:Mmdirder}, while adopting the convention that $\max \emptyset= -\infty$. 

\begin{thm}\label{th:stabgen}
Given system \eqref{eq:diffinc}, an open and connected set $\cD\subset \R^n$ such that $0\in \cD$, and $K$ positive-definite functions $V_1, \dots, V_K \in \cC^1(\cD, \R)$, consider a max-min function $V \in \Mm(V_1, \dots, V_K)$ with $\dot{\overline{V}}_F$ given in \eqref{eq:Mmdirder}.
If there exists a function $\gamma\in \cPd$ such that, for every $x\in \cD$,
\begin{equation}\label{eq:decreasing}
\max \dot{\overline{V}}_F(x) \leq -\gamma(|x|),
\end{equation}
then  $V$ is a Lyapunov function and system \eqref{eq:diffinc} is AS.
If $\cD=\R^n$ and in addition, each $V_j$, $j \in \{1,\dots, K\}$, is radially unbounded, then the origin of \eqref{eq:diffinc} is GAS. 
\end{thm}
A fundamental result for proving Theorem~\ref{th:stabgen} appears in Lemma~\ref{lm:hard} given below. The proof of Lemma~\ref{lm:hard} with some related discussions is deferred to Section~\ref{sec:lemDer}. 
\begin{lemma}\label{lm:hard}
Consider a function $V\in \Mm(V_1,\dots, V_K)$ and a solution $\varphi:[0,T)\to \cD$ of the differential inclusion \eqref{eq:diffinc}. For $t\in [0,T)$,
\begin{subequations}\label{eq:condition}
\begin{align}
&\frac{d}{dt}V(\varphi(t))\,\text{ exists almost everywhere and}\label{eq:ExistenceofDV} \\
&\frac{d}{dt}V(\varphi(t)) \in  \dot{\overline{V}}_F(\varphi(t))\;\text{ almost everywhere.}\label{eq:InclusionofDV}
\end{align}
\end{subequations}
\end{lemma}
\begin{oss}[Comparison with other approaches]\label{rmk:comparisionWithClarke}
In Lemma~\ref{lm:hard}, we relate the Dini derivative of $V$ along the solutions of system~\eqref{eq:diffinc} with the Lie derivative $\dot{\overline V}_F$. In~\cite[Chapter 4.2]{clarke1}, we also see a relationship between the Dini derivative and the directional derivative along vector fields in the context of weak stability. In particular, it is shown that for every $\zeta \in \partial V(\varphi(t))$, $\inf_{\dot \varphi(t) \in F(\varphi(t))} \frac{d}{dt} V(\varphi(t)) \in \inf_{\dot \varphi(t) \in F(\varphi(t))}\zeta^\top \dot \varphi(t)$, for almost every $t \ge 0$. On the other hand, for strong stability, it would be natural to work with the relation, $\sup_{\dot \varphi(t) \in F(\varphi(t))} \frac{d}{dt} V(\varphi(t)) \in \sup_{\dot \varphi(t) \in F(\varphi(t))} \zeta^\top \dot \varphi(t) \le - \gamma(\vert \varphi(t) \vert)$, for every $\zeta \in \partial V(\varphi(t))$. However, such a relation is conservative for our purposes, as it can be seen in Example~\ref{shieldexample} (see Remark~\ref{rem:ExClarkFails}), where the supremum on the right-hand side of the foregoing inclusion is strictly positive along certain directions in the set $F(x)$ for some $x \in \R^2$. The use of Lie derivative in Lemma~\ref{lm:hard} thus provides tighter bounds on the Dini derivative by selecting meaningful directions from the set $F(x)$ for each $x \in \R^n$.
\end{oss}
\begin{oss}
Stability results involving the set-valued Lie derivative~\eqref{eq:genderiv} and condition \eqref{eq:decreasing} are proved in \cite[Proposition 1]{bacciotti99} for locally Lipschitz and  regular (recall Definition~\ref{def:Regular}) Lyapunov functions. Set-valued Lie derivatives are also used in \cite{KamRosPar19} to identify and remove infeasible directions from a differential inclusion when limiting the attention to regular locally Lipschitz functions. Showing that this condition is sufficient when considering locally Lipschitz functions obtained via a max-min composition nontrivially generalizes such results. In fact, a function $V\in\Mm(V_1, \dots, V_K)$ is in general \emph{not} regular: recalling~\eqref{eq:gradConvex}, the definition in~\eqref{eq:regularity} requires, for a regular function $V$, that
\begin{equation}\label{eq:Mmregular}
\begin{aligned}
\lim_{h \to 0^+} \frac{V(x+hv)-V(x)}{h}=\max \{ \nabla V_\ell(x)^\top v \, \vert \, \ell \in \alpha_V(x)\},
\end{aligned}
\end{equation} 
for all $x\in \cD$ and for all $v\in \R^n$.
However considering for example  $V(x)= \min \{V_1(x), V_2(x) \}$, we have that the left-hand side of \eqref{eq:Mmregular} is equal to $\min\{\nabla V_{\ell}^\top v \, \vert \, \ell \in \alpha_V(x)\}$, and thus in general equality \eqref{eq:regularity} doesn't hold. In this sense Lemma \ref{lm:hard} is a generalization of \cite[Proposition 1]{bacciotti99} to a class of nonregular functions.
\end{oss}

Recalling inclusion \eqref{eq:setdiffinclusion}, we can state the following result specifically for \eqref{eq:sysDI}, using the notion of Clarke's generalized derivative, which is generally more conservative than Theorem~\ref{th:stabgen}. This result is also reported in our preliminary conference paper \cite[Theorem 1]{dellarossa18}.
\begin{cor}\label{mainteo}
Consider the DI~\eqref{eq:sysDI}. Given an open and connected set $\cD\subset \R^n$ such that $0\in \cD$ and $K$ positive-definite functions $V_1, \dots, V_K \in \cC^1(\cD, \R)$, consider a max-min function $V \in \Mm(V_1, \dots, V_K)$. Suppose that there exists a function $\gamma\in \cPd$, such that for all $x\in \cD$,
\begin{equation}\label{flow}
\begin{aligned}
\nabla V_\ell(x)^\top f_i(x)\le -\gamma(|x|), \;\; \forall \, \ell\in \alpha_V(x),
\end{aligned}
\end{equation}
for all $i \in \{1, \dots,M\}$. Then the origin of \eqref{eq:sysDI} is AS and $V$ is a Lyapunov function for system~\eqref{eq:sysDI}.
If $\cD= \R^n$, and in addition, each $V_j$, $j \in \{1,\dots, K\}$, is radially unbounded, then the origin of \eqref{eq:sysDI} is GAS.
\end{cor}
\begin{proof}
Consider a point $x \in \cD$, and suppose that $\alpha_V(x)=\{\ell_1, \dots, \ell_p\}$. Recalling Proposition~\ref{lm:gengrad}, for each $v \in \partial V(x)$, there exist $\lambda_1, \dots, \lambda_p \geq 0$, $\sum_{j=1}^p \lambda_j=1$, such that
$
v=\sum_{j=1}^p \lambda_j \nabla V_{\ell_j}(x)$.
Consequently, for each $i\in \{1, \dots,M\}$, \eqref{flow} yields
\begin{align*}
v ^\top f_i(x) & =\sum_{j=1}^p \lambda_j \nabla V_{\ell_j}(x)^\top f_i(x) 
 \leq -\sum_{j=1}^p \lambda_j \gamma(|x|) =-\gamma(|x|),
\end{align*}
which implies that $v^\top  f \leq -\gamma(|x|)$,
for each $v \in \partial V(x)$, and every $f\in \co \bigl \{f_i(x) \, \vert \; i \in \{1, \dots, M\}\bigr\}$. Recalling~\eqref{eq:setdiffinclusion}, inequality~\eqref{eq:decreasing} holds and the result follows from Theorem~\ref{th:stabgen}.
\end{proof}

 While Theorem~\ref{th:stabgen} holds for a general differential inclusion~\eqref{eq:diffinc}, the statement of Corollary~\ref{mainteo} is specifically tailored for system~\eqref{eq:sysDI}. 

\subsection{Proof of Lemma~\ref{lemma:prelimnLemma}}\label{sec:propGrad}
\begin{proof}\textcolor{black}{
\emph{Case 1:} Consider first the case where $V(x)=V_\ell(x)$ for an $\ell \in \{1,\dots, K\}$ and $V(x)\neq V_j(x)$ for all $j\neq \ell$. By continuity of $V,V_1, \dots, V_K$ there exists a neighborhood $\cU\subset \cD$ of $x$ where the non-equality relations are preserved and thus $\cU\subset \inn(C_\ell)$, which implies $x\in \inn(C_\ell)$ and $\alpha_V(x)=\{\ell\}$, in addition to~\eqref{eq:lemmaImplication} with $\ell_z\equiv \ell$.\\
\emph{Case 2:} Let us now consider the general case $V(x)=V_{\ell_1}(x)=\dots=V_{\ell_p}(x)$ and $V(x)\neq V_j(x)$ if $j\notin \{\ell_1,\dots,\ell_p\}$, for some $\ell_1,\dots,\ell_p\in \{1,\dots, K\}$. By continuity of $V,V_1,\dots V_K$ there exists a neighborhood $\cU_0\subset \cD$ of $x$ such that the non-equality relations $V(z)\neq V_j(z)$, $\forall\,j\notin \{\ell_1,\dots,\ell_p\}$ are conserved, for any $z\in \cU_0$. Recalling~\eqref{eq:alfainclusion}, $\alpha_V(x)\subset\{\ell_1,\dots, \ell_p\}$; when $\alpha_V(x)=\{\ell_1,\dots, \ell_p\}$, we are done, by proceeding exactly as in \emph{Case 1}. Otherwise, when $\alpha_V(x)\neq \{\ell_1,\dots, \ell_p\}$ consider without loss of generality, that  $\ell_1\notin \alpha_V(x)$. By Definition~\ref{def:alpha}, $\ell_1\notin \alpha_V(x)$ implies $x\notin \cl(\inn(C_{\ell_1}))$ therefore there exists an open neighborhood $\cU_1$ of $x$ such that
\begin{equation}\label{eq:emptyInter}
\cU_1\subset\cU_0,\;\;\;\;\cU_1\cap \inn(C_{\ell_1})=\emptyset.
\end{equation}
Consider now, if any, each point $\overline x \in \cU_0$ such that $V(\overline x) = V_{\ell_1}(\overline x)$ and $V(\overline x) \neq V_\ell (\overline x)$, for all $\ell \in \{\ell_2, \dots, \ell_p\}$, it again follows from continuity  that $V(z)\neq V_\ell(z)$ for every $\ell\in \{\ell_2,\dots ,\ell_p\}$, and every $z$ in some neighborhood $\cV\subset\cU_0$ of $\overline x$. Moreover we have, by our choice of $\cU_0$, $V(z)\neq V_j(z)$ if $j\notin\{\ell_1,\dots,\ell_p\}$ for every $z\in \cV$, which implies $V(z)=V_{\ell_1}(z)$, $\forall\,z\in \cV$. As a consequence $\overline x \in \inn(C_{\ell_1})$, and by equation \eqref{eq:emptyInter} we have $\overline x\notin\cU_1$. In other words, we have shown that
\begin{equation}\label{eq:LemmaRepresentation}
(z\in \cU_1)\Rightarrow(\exists\,\ell_z\in\{\ell_2,\dots,\ell_p\}\text{ s.t. }V(z)=V_{\ell_z}(z)).
 \end{equation}
 Now, if $\alpha_V(x)=\{\ell_2,\dots,\ell_p\}$,~\eqref{eq:lemmaImplication}~holds with $\cU=\cU_1$ and $\ell_z\in \{\ell_2,\dots,\ell_p\}$. Otherwise we can iterate this argument supposing $\ell_2\notin\alpha_V(x)$ and so on. At each iteration $\nu\in \{2,\dots, p-1\}$, generalizing~\eqref{eq:emptyInter} and~\eqref{eq:LemmaRepresentation}, we construct an open neighborhood $\cU_\nu$ of $x$ such that $\cU_\nu\subset \cU_{\nu-1}$ and 
 \begin{equation}\label{eq:LemmaLastInclusion}
(z\in \cU_\nu)\Rightarrow(\exists\,\ell_z\in\{\ell_{\nu+1},\dots,\ell_p\}\text{ s.t. }V(z)=V_{\ell_z}(z)).
 \end{equation}
 Either $\alpha_V(x)=\{\ell_{\nu+1}, \dots, \ell_p\}$ and the proof is complete with $\cU=\cU_\nu$ or we need to iterate again. Note that when $\nu=p-1$, the existence of $\ell_z\in \{\ell_p\}$ as in~\eqref{eq:LemmaLastInclusion}, implies $V(z)=V_{\ell_p}(z)$, for all $z\in \cU_{p-1}$, thus proving $\cU_{p-1}\subset\inn(C_{\ell_p})$ and hence $\alpha_V(x)=\{\ell_p\}$. This completes the proof of \eqref{eq:lemmaImplication} and the fact that $\alpha_V$ is non-empty.}
\end{proof}

\subsection{Proof of Lemma~\ref{lm:hard}}\label{sec:lemDer}
Lemma~\ref{lm:hard} is the key result used in the proof of Theorem~\ref{th:stabgen}, establishing properties of the directional  derivative of  $V \in \Mm(V_1, \dots, V_K)$ along the solutions of~\eqref{eq:diffinc}. In its proof we will use the following result.
\begin{claim}\label{Claim:Exchanging}
Given functions $\xi_1,\dots \xi_J:\R\to \R$ continuous at $0$, we have that 
\[
\lim_{h\to 0}\min_{j\in\{1,\dots,J\}}\xi_j(h)=\min_{j\in \{1,\dots, J\}} \lim_{h\to 0}\xi_j(h).
\]
\end{claim}
\begin{proof}[Proof of Claim~\ref{Claim:Exchanging}] Define $\xi(h):=\min_{j\in \{1,\dots, J\}}\xi_j(h)$ for all $h\in \R$; $\xi$ is continuous at $0$ since it is the pointwise minimum of continuous functions. We have
\[
\begin{aligned}
\lim_{h\to 0}&\min_{j\in\{1,\dots,J\}}\xi_j(h)=\lim_{h\to 0}\xi(h)=\xi(0)\\=&\min_{j\in\{1,\dots,J\}}\xi_j(0)=\min_{j\in \{1,\dots, J\}} \lim_{h\to 0}\xi_j(h),
\end{aligned}
\]
thus concluding the proof.
\end{proof}
\begin{proof}[Proof of Lemma~\ref{lm:hard}]
Recalling that $\varphi( \cdot)$ is an absolutely continuous solution of the differential inclusion \eqref{eq:diffinc} and that $V$ is a locally Lipschitz function, the function $V \circ \varphi : [0,T) \to \R$ is absolutely continuous, and hence $\frac{d}{dt}V(\varphi(t))$ exists almost everywhere in $[0,T)$, proving~\eqref{eq:ExistenceofDV}. Moreover, there exists a set $\cN_0$ of measure zero such that, for every $t \in [0,T) \setminus \cN_0$, both $\dot \varphi(t)$ and $\frac{d}{dt}V(\varphi(t))$ exist, and $\dot \varphi(t) \in F(\varphi(t))$.\\
To prove~\eqref{eq:InclusionofDV}, from Proposition~\ref{prop:Equivalence}, we use representation \eqref{eq:MinMax} of $V\in \Mm(V_1, \dots, V_K)$, dropping the superscript ``$\star$'' for notational simplicity, that is
\[
V(x):=\min_{j \in \{1, \dots, J\}} \left \{\max_{\ell \in S_j} \{V_\ell(x)\} \right \},
\]
where $J\geq 0$ and  $S_1, \dots, S_J$ are non-empty subsets of $\{1, \dots, K\}$. 
By Lemma~\ref{lemma:prelimnLemma}, for each $t \in [0,T)$, and for each $x$ in a neighborhood of $\varphi(t)$, $V(x)$ can be expressed as
\[
V(x):=\min_{j \in \{1, \dots, J\}} \left \{\max_{\ell \in S_j \cap \alpha_V(\varphi(t)) } \{V_\ell(x)\} \right \};
\]
namely only the active indexes in $\alpha_V(\varphi(t))$ play a role (possibly ruling out the sets $S_j$ for which $S_j\cap \alpha_V(\varphi(t))=\emptyset$).
Let us introduce the notation
\begin{equation}\label{eq:VBar}
\oV_j(x):=\max_{\ell \in S_j\cap \alpha_V(\varphi(t)) } \{V_\ell(x)\}.
\end{equation}
To proceed in a constructive manner, consider the set $\mathbf{M}(V_1,\dots, V_K)$ containing all the functions obtained by max (and only max) combination over $V_1, \dots, V_K$. The cardinality of $\mathbf{M}(V_1,\dots, V_K)$ is finite and equal to $N_K:=2^K-1$ and we can denote its elements by $W_k$, for $k \in \{1, \dots, N_K\}$. Reasoning as before, for each $k$ define  $\cN_k$ as the subset of $[0,T)$ where $W_k \circ \varphi$ is not differentiable. Since $W_k$ are locally Lipschitz, then each $\cN_k$ has measure zero.
Fix any $t \in [0,T) \setminus (\bigcup_{k\in \{0, \dots, N_K\}}\cN_k)$.
From the fact that $\oV_j$ in~\eqref{eq:VBar} is locally Lipschitz for each $j \in \{1, \dots, J\}$, we obtain
\begin{equation}\label{eq:derder}
\frac{d}{dt}\overline{V}_j(\varphi(t))= \lim_{h \to 0} \frac{\oV_j(\varphi(t)+h \dot \varphi(t))-\oV_j(\varphi(t))}{h},
\end{equation}
where the limit exists because $t \notin \bigcup_{k\in \{1, \dots, N_K\}}\cN_k$.
The functions $\overline{V}_j$ in~\eqref{eq:VBar} are regular (Definition~\ref{def:Regular}). We can follow the idea of \cite[Lemma 1]{bacciotti99}:
by letting $h$ go to zero from the right, recalling inclusion \eqref{eq:alfainclusion}, we get
\begin{align}\label{eq:derderri}
\frac{d}{dt}\overline{V}_j(\varphi(t))=\max_{\ell \in S_j \cap\alpha_{V}(\varphi(t))} \left \{\nabla V_\ell(\varphi(t))^\top \dot \varphi(t)\right \}.
\end{align}
Similarly, by letting $h$ go to zero from the left in \eqref{eq:derder}, we get
\begin{align}\label{eq:derderle}
\frac{d}{dt}\overline{V}_j(\varphi(t))=\min_{\ell \in S_j \cap\alpha_{V}(\varphi(t))} \left \{\nabla V_\ell(\varphi(t))^\top \dot \varphi(t)\right \}.
\end{align}
Since $\frac{d}{dt}\oV_j(\varphi(t))$ exists, we have \eqref{eq:derderri}=\eqref{eq:derderle}, and thus for each $j \in \{1, \dots, J\}$ we can write, for all $\ell \in S_j \cap \alpha_V(\varphi(t))$,
\begin{equation}\label{eq:ajdef}
\frac{d}{dt}\oV_j(\varphi(t))=\nabla V_{\ell}(\varphi(t))^\top \dot \varphi(t)=:a_j(t).
\end{equation}
Now consider the function $V(x)=\min_{j \in \{1, \dots, J\}} \left \{\oV_j(x) \right \}$, for $x$ in some neighborhood of $\varphi(t)$. For all $h>0$, we use the fact that  $\oV_j(\varphi(t))=V(\varphi(t))$ for all $j\in \{1,\dots, J\}$, to obtain
\[
\begin{aligned}
\xi(h)&:=\frac{V(\varphi(t)+h\dot \varphi(t))-V(\varphi(t))}{h}\\
&=\frac{\min_{j}\{\oV_j(\varphi(t)+h\dot \varphi(t))\}-V(\varphi(t))}{h}\\
&=\min_{j \in \{1, \dots, J\}}\left\{\tfrac{\oV_j(\varphi(t)+h\dot \varphi(t))-\oV_j(\varphi(t))}{h} \right\} =:\min_{j \in \{1, \dots, J\}}\xi_j(h).
\end{aligned}
\]
Then, applying Claim~\ref{Claim:Exchanging} and~\eqref{eq:ajdef} we have
\begin{align}
\frac{d}{dt}V(\varphi(t))&=\lim_{h\to0^+}\min_{j \in \{1 \dots J\}}\left\{\xi_j(h) \right\}=\hskip-0.2cm\min_{j\in \{1 \dots,J\}}\hskip-0.1cm\left \{\lim_{h\to 0^+}\xi_j(h)\right\}\nonumber
\\ 
&=\min_{j\in \{1, \dots, J\}}\left\{\frac{d}{dt}\oV_j(\varphi(t))\right\}=\min_{j\in \{1, \dots, J\}}\{a_j(t)\}\label{eq:lemmaprop1}.\end{align}
Using again Claim~\ref{Claim:Exchanging}, we can also write 
\begin{align}
\frac{d}{dt}&V(\varphi(t))=\lim_{h\to 0^-}\frac{V(\varphi(t)+
h\dot \varphi(t))-V(\varphi(t))}{h}\nonumber\\
&=-\lim_{h\to 0^-}\left (\min_{j\in \{1, \dots, J\}}\left \{\frac{\oV_j(\varphi(t)+
h\dot \varphi(t))-\oV_j(\varphi(t))}{-h}\right \}\right) \nonumber \\
&=-\min_{j\in \{1, \dots, J\}}\left \{\lim_{h\to 0^-}\frac{\oV_j(\varphi(t)+
h\dot \varphi(t))-\oV_j(\varphi(t))}{-h}\right\}\nonumber\\
&=-\min_{j\in \{1, \dots, J\}}\{-a_j(t)\}=\max_{j\in \{1, \dots, J\}}\{a_j(t)\}. \label{eq:lemmaprop2}
\end{align}
Summarizing, from \eqref{eq:lemmaprop1} and \eqref{eq:lemmaprop2}, it follows that $a_1(t)=\dots=a_J(t):=\overline{a}(t)$. Therefore, from \eqref{eq:ajdef} we get, for each $j\in \{1,\dots,J\}$, that $\ell\in S_j\cap\alpha_V(\varphi(t))$ implies $\nabla V_\ell(\varphi(t))^\top \dot \varphi(t)=\overline{a}(t)$. Finally, recalling that $\alpha_V(\varphi(t))=\bigcup_{j} S_j\cap \alpha_V(\varphi(t))$, we have
\begin{equation*}
 \nabla V_{\ell}(\varphi(t))^\top \dot \varphi(t)=\overline{a}(t),\;\;\forall \ell \in \alpha_V(\varphi(t)).
 \end{equation*}
 \begin{figure*}[t!]
    \centering
    \begin{subfigure}[t]{0.5\textwidth}
        \centering
        \includegraphics[height=3.6cm]{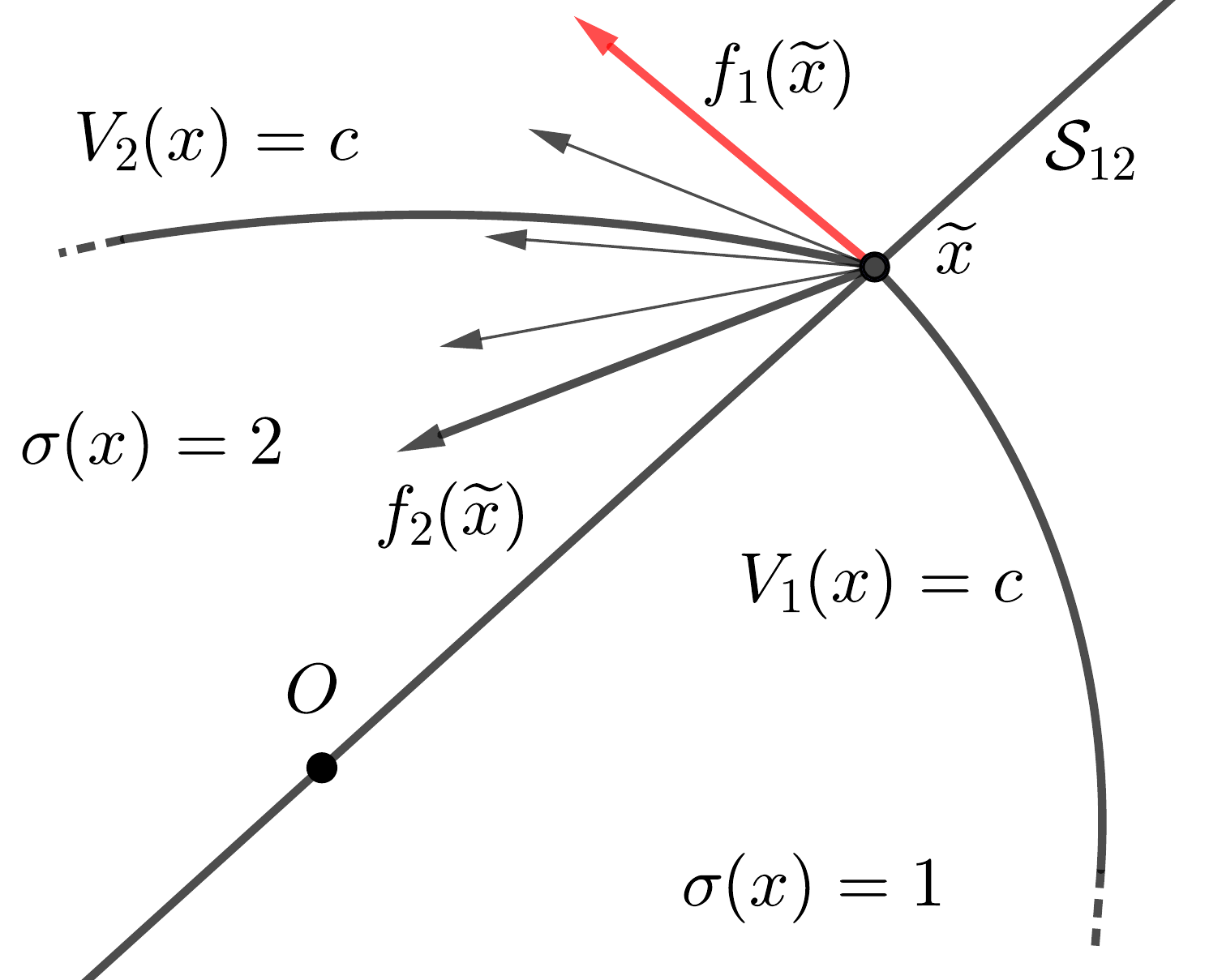}
        \caption{The vector fields $f_1(\widetilde{x})$ and $f_2(\widetilde{x})$  are pointing in the same half-plane, which corresponds to the case $\dot{\overline{V}}_{F^\sw}(\widetilde{x})=\emptyset$.}
    \end{subfigure}%
    ~ 
    \begin{subfigure}[t]{0.5\textwidth}
        \centering
        \includegraphics[height=3.6cm]{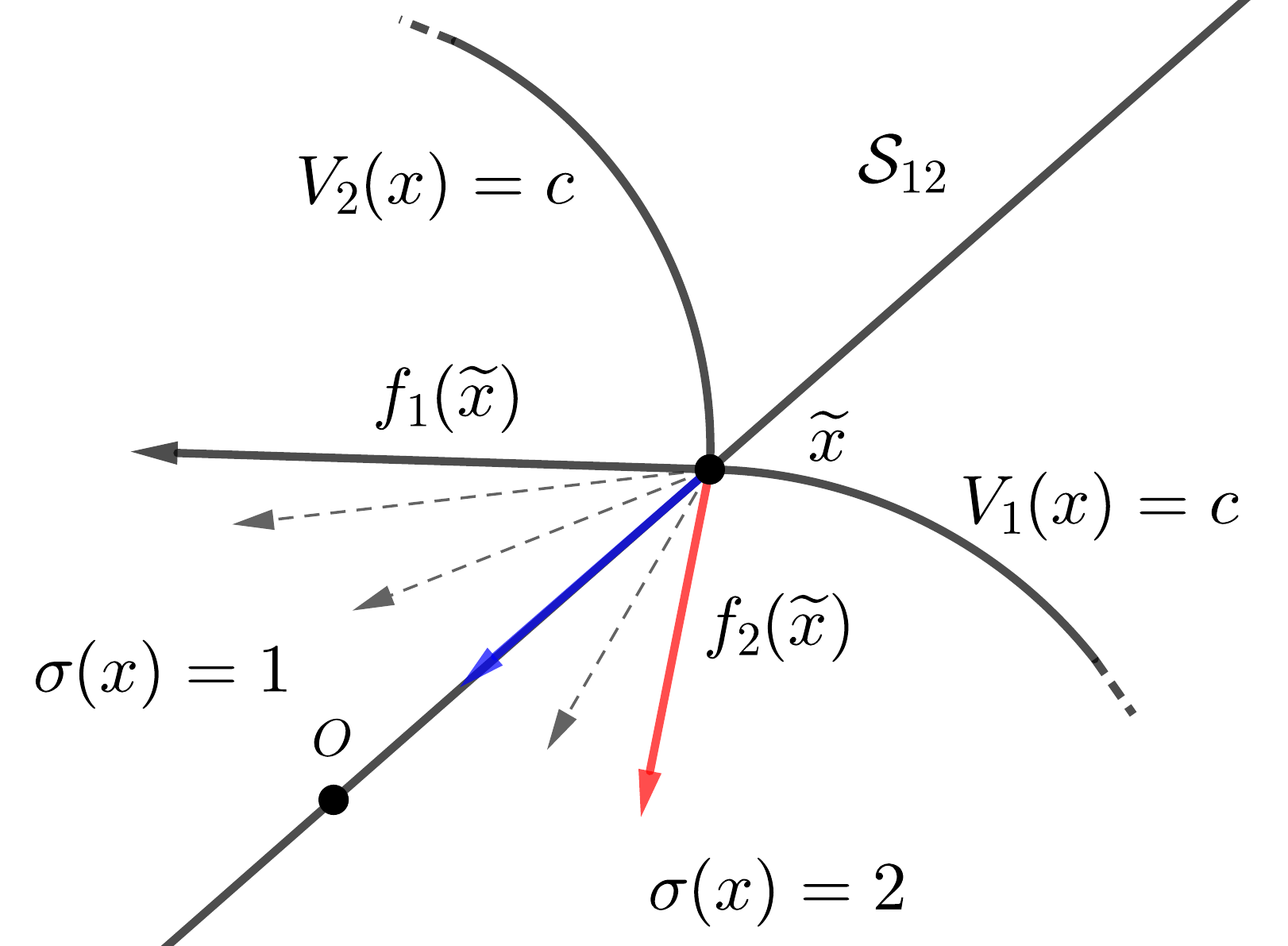}
        \caption{ A convex combination of the vector fields $f_1(\widetilde{x})$ and $f_2(\widetilde{x})$ aligns with the tangent space of $\cS_{12}$ at $\wt x$  and thus, $\dot{\overline{V}}_{F^\sw}(\widetilde{x}) \neq \emptyset$.}
        \label{figure:nongenericTANG}
    \end{subfigure}
    \caption{A geometric interpretation of the set $\dot{\overline{V}}_{F^\sw}(\widetilde{x})$ in $\R^2$.}
    \label{figure:tangenttruc}
\end{figure*}
From~\eqref{eq:Mmdirder}, it follows that $\overline{a}(t)\in \dot{\overline{V}}_F(\varphi(t))$, which then implies~\eqref{eq:InclusionofDV}.
\end{proof}

\section{Switched systems}\label{sec:swSys}
We now focus our attention on system  \eqref{switching}. In contrast to \eqref{eq:sysDI}, where the vector fields may switch to any value at any point in the state space, the switching in \eqref{switching} occurs according to the pre-specified function $x\mapsto\sigma(x)$, which determines the active vector field as a function of the state. As a consequence, solutions of \eqref{switching} are also solutions of \eqref{eq:sysDI} and Theorem \ref{th:stabgen} also implies GAS of \eqref{switching}. However we search here for less conservative stability conditions.
Let $f_1, \dots, f_M$ be $\mathcal{C}^1(\mathbb{R}^n ,\mathbb{R}^n)$ in \eqref{switching}. The class of switching functions $x\mapsto \sigma(x)$ that we consider for system~\eqref{switching} is introduced in the following assumption.
\begin{assumption}\label{ass:swSig}
There exist finitely many analytic functions $H_1,\dots,H_M:\R^n\to \R$, defining open sets $D_1, \dots, D_M \subset \R^n$ by
\[
D_i : = \{x \in\R^n \, \vert \, H_i(x) > 0\},\;\;\;\;\;\forall\,i\in\{1,\dots ,M\},
\]
such that $\sigma$ is constant and equal to $i$ on each $D_i$, 
\[
\bigcup_{i=1}^M \overline{D_i}= \mathbb{R}^n, \quad \text{ and }\; D_i\cap D_j=\emptyset, \;\;\text{if}\;i\neq j.
\]
\end{assumption}
Note that, in Assumption~\ref{ass:swSig} the value of $\sigma$ remains unspecified on $\partial D_i$, i.e. the boundaries of $D_i$, $i=\{1,\dots, M\}$. Since $\partial D_i\subset \{x\in \R^n\;\vert \;H_i(x)=0\}$, and the set of zeros of an analytic function has zero Lebesgue measure, this ambiguity will not affect the solution set of~\eqref{switching}, as explained in the sequel.

Given $f_1, \dots, f_M \in \mathcal{C}^1(\mathbb{R}^n ,\mathbb{R}^n)$ and $\sigma: \mathbb{R}^n \to \{1, \dots, M\}$  satisfying Assumption~\ref{ass:swSig}, we define $f^\sw:\mathbb{R}^n \to \mathbb{R}^n$, as
\begin{equation}\label{definsw}
f^\sw(x):=f_{\sigma(x)}(x).
\end{equation}
Because the vector field in \eqref{definsw} is in general discontinuous, we define an appropriate notion of solution of \eqref{definsw}, arising from  the Filippov regularization.

 \begin{figure*}[t!]
 \begin{subfigure}[t!]{0.3\textwidth}
 \centering
 \includegraphics[scale=0.66]{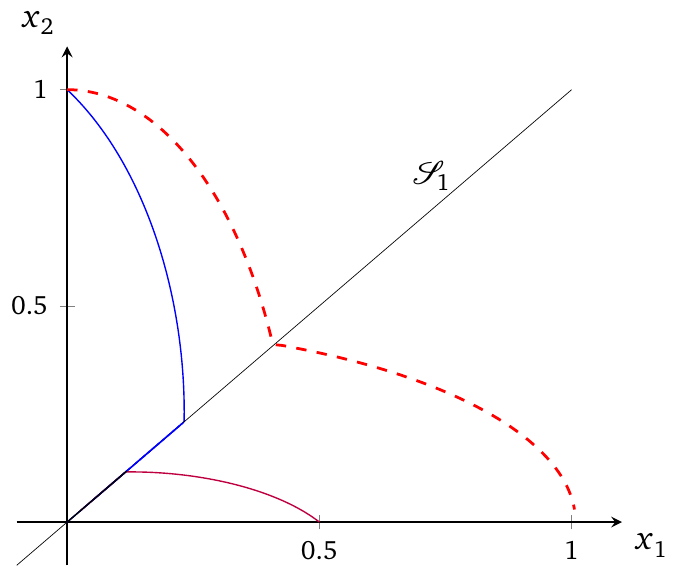}
 \caption{The blue line shows a trajectory starting from $(0,1)$, the red line a trajectory starting from $(0.5,0)$ and the  red dashed line indicates a level set of $V(x)$. }
 \label{figure:nonLinsys}
 \end{subfigure}%
 \hspace{1em}
 \centering
 \begin{subfigure}[t!]{0.3\textwidth}
 \includegraphics[scale=0.66]{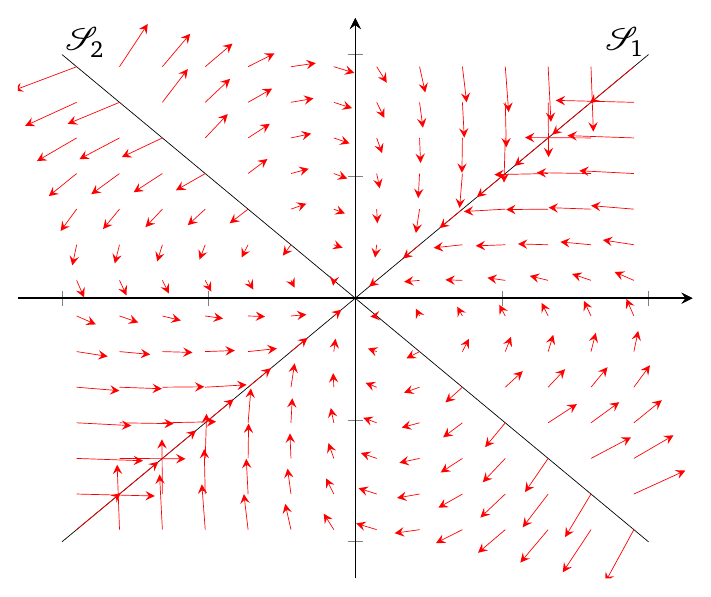}
 \caption{The red arrows represent the vector field on the whole state-space. Let us note the converging sliding motion on the line $\mathcal{S}_1$ and the nongeneric case on the line $\mathcal{S}_2$.}
 \label{figure:nongeneric}
 \end{subfigure}
 \hspace{1em}
 \begin{subfigure}[t!]{0.30\textwidth}
 \includegraphics[scale=0.66]{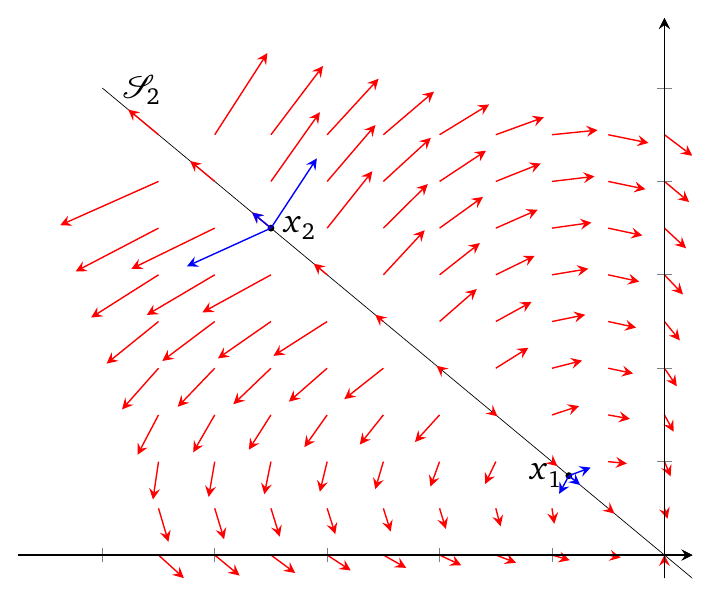}
 \caption{The blue arrows represent the elements of $F_{f^{sw}}(x)$, and in particular the convex combination of $A_1x$ and $A_2 x$ that is pointing toward 0 near the origin and diverging away from the origin.}
 \label{figure:nongeneric2}
 \end{subfigure}
 \caption{Trajectories of switched system~\eqref{eq:flowereverse} in Example~\ref{ex:inverseflower}.}
 \end{figure*}

\begin{defn}[\cite{filippov1988differential}]
Given $f^\sw:\mathbb{R}^n \to \mathbb{R}^n$ in~\eqref{definsw}, and the system 
\begin{equation}\label{noncont}
\dot x(t)= f^\sw(x(t)), 
\end{equation}
define the set-valued \textit{Filippov regularization}
\begin{equation}\label{filippov}
\dot x \in F^\sw(x):= \bigcap_{\varepsilon>0} \bigcap_{\substack{\mathcal N\subset \R^n, \\\mu(\mathcal N)=0}} \overline{\co}\bigl \{f^\sw(\B_{\varepsilon}(x) \setminus \mathcal N)\bigr \},
\end{equation}
where $\mu(\mathcal N)$ is the Lebesgue measure of $\mathcal N \subset \mathbb{R}^n$ and $\overline{\co}$ denotes the closed convex hull.
We say that $x: \mathbb{R}_{\geq0} \to \mathbb{R}^n$ is a \textit{Filippov solution} of  system (\ref{noncont}) starting at $x_0$ if 
\begin{enumerate}
\item $x$ is absolutely continuous, with $x(0)=x_0$,
\item $\dot x(t) \in F^\sw(x(t))$ for almost all $t>0$.
\end{enumerate}
\end{defn}
For the vector field $f^\sw$ in \eqref{definsw}, the computation of $F^\sw$ is simplified as observed in \cite[Page~51]{cortes} and is summarized below:
\begin{prop}\label{pr:abcde}
Consider the vector field $f^{\sw}$ in \eqref{definsw} with $\sigma$ satisfying Assumption~\ref{ass:swSig}. Introduce the set-valued map $I:\R^n \rightrightarrows \{1,\dots, M\}$ as
\begin{equation}\label{eq:defJ}
I(x) := \{i\;\vert\;x\in \overline{D_i}\}=\bigcap_{\varepsilon>0} \bigcup_{\substack{y\in \B_\varepsilon (x)\\y\in \bigcup_i D_i}} \sigma(y).
\end{equation}
Then $F^\sw$ in~\eqref{filippov} satisfies
\begin{equation}\label{eq:swFilippov}
F^\sw(x) = \co\{f_i(x)\, \vert \, i \in I(x) \}.
\end{equation}
\end{prop}

We underline that under Assumption \ref{ass:swSig} the Filippov regularization $F^\sw$ is an upper semi-continuous map with $F^\sw(x)$ being nonempty, compact, and convex for each $x \in \R^n$. Thus, we can study stability of switched systems in~\eqref{definsw} using the results developed in Section \ref{sec:maxmin}. Defining $\dot{\overline{V}}_{F^\sw}(x)$ as in \eqref{eq:Mmdirder} with $F$ replaced by $F^\sw$,  Theorem~\ref{th:stabgen} leads to the following statement in the context of switched systems.
\begin{thm}\label{thm:switchstab}
Consider system \eqref{switching}, and a switching law $\sigma:\mathbb{R}^n \to \{1, \dots, M\}$ satisfying Assumption~\ref{ass:swSig}. Consider an open and connected set $\cD\subset \R^n$ such that $0\in \cD$ and $K$ positive-definite functions $V_1, \dots, V_K \in \cC^1(\cD, \R)$. If, for a max-min  function $V\in\Mm \{V_1, \dots, V_K\}$, there exists $\gamma \in \cPd$ such that
\begin{equation}\label{eq:prop1ineq}
\begin{aligned}
\max \dot{\overline{V}}_{F^\sw}(x)\leq -\gamma(|x|),\;\;\forall x\in\cD,
\end{aligned}
\end{equation}
then the origin of \eqref{filippov} is AS. If $\cD= \R^n$, and each $V_j$, $j \in \{1, \dots,K\}$, is radially unbounded, then \eqref{filippov} is GAS.
\end{thm}
Theorem~\ref{thm:switchstab} simultaneously accounts for points $x$ where $I(x)$ (associated to $\sigma$), and/or points where $\alpha_V(x)$ (associated to $V$) are multivalued. Interesting things happen when these points coincide, namely when $V$ mimics the patchy shape of $F^\sw$.

\begin{oss}\label{oss:construction}
Consider the simplest non-trivial case, taking an $\wt x\in \cD$ such that $I(\wt x)=\{1,2\}$ and $\alpha_V(\wt x)=\{\ell_1, \ell_2\}$, for some $\ell_1, \ell_2 \in \{1, \dots, K\}$. We may give a geometric interpretation of \eqref{eq:prop1ineq}. Parameterizing an $f\in F^\sw (\wt x)$ with $f=\lambda f_1(\wt x)+(1-\lambda)f_2(\wt x)$ in expression~\eqref{eq:Mmdirder}, we have that  $\dot{\overline{V}}_{F^{sw}}(\wt x) \neq \emptyset$, if and only if there exists $\lambda \in [0,1]$ such that (we omit the argument $\wt x$ of the gradients to simplify the notation),
\begin{equation*}
\lambda (\nabla V_{\ell_1} - \nabla V_{\ell_2})^\top f_1(\wt x) =- (1-\lambda) (\nabla V_{\ell_1} - \nabla V_{\ell_2})^\top f_2(\wt x),
\end{equation*}
which holds  only if
\begin{equation*}
\left ((\nabla V_{\ell_1} - \nabla V_{\ell_2})^\top f_1(\wt x)\right)\left((\nabla V_{\ell_1} - \nabla V_{\ell_2})^\top f_2(\wt x)\right)\leq 0.
\end{equation*}
It follows that $\dot{\overline{V}}_{F^{sw}}(\wt x) \neq \emptyset$ only if the vector fields $f_1(\wt x)$ and $f_2(\wt x)$ are such that the inner product of their respective components, normal to the hypersurface $\cS_{12}=\{ x \in \R^n \, \vert \, V_{\ell_1}(x)=V_{\ell_2}(x) \}$ is negative, namely they do not point both on the same side of $\cS_{12}$. Figure~\ref{figure:tangenttruc} provides an illustration of this fact in the planar case.
\end{oss}

In Example~\ref{ex:inverseflower}, an illustration of this idea is provided.

 \begin{myexample}
 \label{ex:inverseflower}

 We consider a system of the form \eqref{switching} and analyze its stability using~Theorem~\ref{thm:switchstab}.
 Given
 $A_1= \begin{bsmallmatrix}
   -0.1 & 1 \\
    -5 & -0.1
  \end{bsmallmatrix}$, $A_2= \begin{bsmallmatrix}
     -0.1 & -5 \\
    1 & -0.1
  \end{bsmallmatrix}$
 and $Q=\begin{bsmallmatrix}
    1 & 0 \\
     0 & -1
  \end{bsmallmatrix}$,
 consider the switched system
 \begin{equation}
 \label{eq:flowereverse}
 \dot x=\begin{cases} 
 f_1(x):=A_1x - b \wt g(x), & \text{if}\;\; x^\top Qx < 0,\\
 f_2(x):=A_2x - b \wt g(x), & \text{if}\;\; x^\top Qx > 0,
 \end{cases}
 \end{equation}
 where $b \geq 0$, and  function $\wt g : \R^2 \to \R^2$ is defined as
 \begin{equation*}
 \wt g(x_1, x_2)= \begin{bmatrix}
  g(x_1)\\
 g(x_2)
 \end{bmatrix} = \begin{bmatrix}
  \arctan(x_1)\\
 \arctan(x_2)
 \end{bmatrix}.
 \end{equation*}
 System~\eqref{eq:flowereverse} can be written as~\eqref{definsw}, and satisfies Assumptions~\ref{ass:swSig} with $H_2(x)=-H_1(x)=x^\top Q x$.
Consider now $
 P_1=
 \begin{bsmallmatrix}
 5 & 0 \\
    0  &  1
   \end{bsmallmatrix},$ $ P_2=
 \begin{bsmallmatrix}
 1 & 0 \\
    0  &  5
   \end{bsmallmatrix}$,
we prove that $V(x)= \min \{x^\top P_1 x, x^\top P_2 x\}$ is a Lyapunov function in the sense of Theorem~\ref{thm:switchstab}.
 Noting that $P_1-P_2=4Q$, we can say that the points where $V$ is not differentiable coincide with the points where $\sigma$ is not continuous. To show inequality \eqref{eq:prop1ineq}, we proceed in three steps:\\
 \emph{Step~1: Each subsystem is GAS.}
Analyzing each subsystem where $V$ is differentiable, it can be shown that
 \begin{equation*}
 \nabla V(x)^\top f \leq -0.1|x|^2,\;\, \forall f \in F^\sw(x), \quad \text{if } x^\top Qx \neq 0.
 \end{equation*}
 The next step is to check the inequality \eqref{eq:prop1ineq} where $V$ is not differentiable, that is on the lines $\cS_1 :=\{ x \in \R^2 \; \vert \; x_2=x_1 \},$ and $\cS_2 :=\{x \in \R^2 \; \vert \; x_2=-x_1 \}$, so that $\cS_1 \cup \cS_2$ is the set where $x^\top Q x=0$.\\
 \emph{Step~2: Line $\cS_1$ with converging sliding motion.} We compute the set-valued derivative $\dot{\overline{V}}_{F^\sw}(x)$ for a point $x \in \cS_1$. Proceeding as in Remark~\ref{oss:construction}, based on~\eqref{eq:Mmdirder}, it is seen that
 \begin{equation}\label{eq:gradS10}
  \lambda x^\top (P_1-P_2)f_1(x)+ (1-\lambda) x^\top (P_1-P_2)f_2(x)= 0
 \end{equation}
 holds with $\lambda = 0.5$, for every $x \in \cS_1$. Consequently, for each $x \in \cS_1$, we have 
\[ \dot{\overline{V}}_{F^\sw}(x)=\left \{2 x^\top P_1 \left (\frac{1}{2}f_1(x)+\frac{1}{2}f_2(x)\right)\right\}
 = \left \{ x^\top P_1 \left(A_1 x+A_2 x\right)- 2 b\,x^\top P_1 \wt g(x)\right\}
\]. By construction, the same singleton would be obtained if we replaced $P_1$ by $P_2$.
Substituting the values of $A_i$ and $P_i$, $i = 1,2$, it thus follows that $\max \dot{\overline{V}}_{F^\sw}(x)<-\frac{25}{2}|x|^2$, $\forall x \in \cS_1$.
In Figure \ref{figure:nonLinsys}, we have plotted two converging \virgolette{sliding} solutions.\\
\emph{Step~3: Line $\cS_2$ with diverging sliding motion.}
Choosing $x\in \cS_2$, and following the same reasoning as in Step~2, it is seen that the set $\dot{\overline{V}}_{F^\sw}(x)$ is nonempty because \eqref{eq:gradS10} holds with $\lambda = 0.5$, for every $x \in \cS_2$. As a result, $\dot{\overline{V}}_{F^\sw}(x)= \left \{ x^\top P_1 (A_1 x+A_2 x)+2b\,x^\top P_1 \wt g(x)\right\}$.
Analyzing the linear term, we have 
 $
 x^\top P_1 (A_1 x+ A_2 x)=22.8\,x_1^2
 $
 ; for the nonlinear term, for each $x \in \cS_2$, we have
 \[
 \begin{aligned}
 -2b \, x^\top P_1 \wt g(x)= -12 b \, x_1 g(x_1).
 \end{aligned}
 \]
 For $x_1$ small enough, we see that 
 $
 x_1 g(x_1) =x_1 \arctan(x_1)= x_1^2 + o(x_1^2)
 $
 where $\lim_{x_1 \to 0} \frac{o(x_1^2)}{x_1^2} = 0$. Thus, for sufficiently large values of $b > 0$, there exists a $\delta>0$ such that
 \begin{equation}\label{eq:Vbardotex}
 \dot{\overline{V}}_{F^\sw}(x) = 22.8\,x_1^2-12bx_1^2 + o(x_1^2) <- 0.1|x|^2,
 \end{equation}
 if $x\in \cS_2$ and $|x|<\delta$.\\
 Combining the three steps we proved~\eqref{eq:prop1ineq} (with $\gamma(|x|)=0.1|x|^2$) for a small open neighborhood $\cD=\B_\delta(0)$ of the origin, and Theorem~\ref{thm:switchstab} establishes local asymptotic stability of the origin by using the minimum of two quadratics as a Lyapunov function.
 Condition \eqref{eq:prop1ineq} fails to be true on the line $\cS_2$, away from the origin regardless of the selection of $b>0$. Hence, there exist Filippov solutions, starting in $\cS_2$ with large enough initial condition that stay in $\cS_2$ and diverge; see Figures \ref{figure:nongeneric} and \ref{figure:nongeneric2} for an illustration.
 We want to underline that, since $x^\top (A_1^\top P_2+A_1P_2)x>0$ \emph{for all} $x \in \cS_2$, recalling~\eqref{eq:oldgenderiv}, it holds that $\max \dot V_F(x)>0$, $\forall \;x\in \cS_2$. This observation again shows the utility of using Lie derivative compared to Clarke's derivative in~\eqref{eq:oldgenderiv}, which does not allow establishing asymptotic stability of the origin.
 
 \end{myexample}

\section{Linear Switched Systems and Quadratic Basis}\label{sec:linear}

We are now interested in applying Theorem~\ref{thm:switchstab} to switched  systems~\eqref{definsw} with \emph{linear} vector fields and a partition given by symmetric cones. More precisely, given $A_1, \dots, A_M\in \R^{n \times n}$, we consider the differential inclusion
\begin{equation}\label{eq:linear}
\dot x \in F_{\text{lin}}^\sw(x) := \co\{A_ix\, \vert \, i \in I(x) \}.
\end{equation}
The set valued map $I:\R^n \rightrightarrows \{1,\dots, M\}$ arises from a switching function $x\mapsto \sigma(x)$ satisfying Assumption~\ref{ass:swSig}, where the sets $D_1,\dots, D_M\subset\R^n$ are defined by
\begin{equation}\label{eq:SwitchingCones}
D_i:=\{x\in \R^n\;\vert\;x^\top Q_ix>0\},
\end{equation}
with properly chosen symmetric matrices $Q_i\in \Sym(\R^{n}):=\{R\in\R^{n\times n}\;\vert\;R^\top=R\}$ and $Q_i$ not negative semidefinite for each $i\in\{1,\dots ,M\}$.
The sets $D_i$ in~\eqref{eq:SwitchingCones} are \emph{symmetric open cones} (if $x\in D_i$ then $\lambda x\in D_i$ for all $\lambda\in \R\setminus\{0\}$).
The map $I:\R^n\rightrightarrows\{1,\dots,M\}$ in~\eqref{eq:defJ}, can be rewritten in this context as follows:
\[
I(x):=\{i\in \{1,\dots,M\}\;\vert\;x^\top Q_ix\geq0 \}.
\]
Indeed $Q_i$ not negative semidefinite implies $\overline{D_i}=\{x\in \R^n\,\vert\,x^\top Q_ix\geq0 \}$.
\begin{oss}\label{rmk:StatePartition} Another possible kind of partition of the state space arises by considering \emph{polyhedral cones} (with a common vertex at the origin), that is sets $D_1,\dots, D_M\subset\R^n$ (satisfying Assumption~\ref{ass:swSig}) defined by linear inequalities
$
D_i:=\{x\in \R^n\;\vert K_ix\geq_c 0\}$, where $K_i\in \R^{k_i\times n}$, for all $i\in \{1,\dots, M\}$ and $\geq_c$ denotes the component-wise relation. The techniques employed in what follows could be adapted also to this case.
\end{oss}
We restrict our attention to  Lyapunov functions \emph{homogeneous of degree 2}, considering max-min functions obtained from quadratic forms. This choice is motivated by the fact that, as proved in~\cite{HuBlanchini2010}, max of quadratics Lyapunov functions are universal (existence is sufficient \emph{and necessary}) for GAS of linear differential inclusions (LDI). For linear state-dependent switched systems~\eqref{eq:linear}, as we noted, non-convex (but still homogeneous) Lyapunov functions are required, and thus the \emph{min}-operator was added to have this flexibility. The study of universality for max-min of quadratics for~\eqref{eq:linear} is open for further research. The construction of ``piecewise'' quadratic Lyapunov functions, in similar settings, is studied also in \cite[]{johansson}, \cite[]{goebel}, and references therein.

\begin{defn}
Given $K$ distinct, symmetric and positive definite matrices $P_1, \dots P_K \in \mathbb{R}^{n \times n}$, a \textit{max-min of quadratics} is denoted by $V \in \Mmq(P_1,\dots,P_K)$, and is defined as
\begin{equation}\label{Mm}
V(x)=\max_{j \in\{1, \dots, J\}} \left \{\min_{k \in S_j} \left\{x^\top P_k x\right\} \right \},
\end{equation}
where $J\geq 1$ and for each $j \in \{1, \dots, J \}$, the set $S_j \subset \{1, \dots, K\}$ is nonempty.
\end{defn}

\begin{oss}[Homogeneity]\label{rmk:Homogeneity}
Since the sets $D_i$ are symmetric cones, the set-valued map in~\eqref{eq:linear} is \emph{homogeneous of degree~1}, in the sense that
$
F^\sw_{\text{lin}}(\lambda x)=\lambda F^\sw_{\text{lin}}(x), \;\forall x\in \R^n, \forall \lambda\in \R$. \\
Similarly, a max-min of quadratics function defined as in~\eqref{Mm} is \emph{homogeneous of degree 2}, that is
$
V(\lambda x)=\lambda^2V(x),\;\forall x\in \R^n, \forall \lambda\in \R
$,  
 and $\alpha_V$ is constant along rays emanating from the origin, that is
$
\alpha_V(\lambda x)=\alpha_V(x),\;\;\forall x\in \R^n,\;\forall \lambda \in \R\setminus\{0\}.
$
\end{oss}

\subsection{Stability Conditions with Set-Valued Lie Derivative}
We first specialize the conditions of Theorem~\ref{thm:switchstab} for system~\eqref{eq:linear} with $V$ of the form \eqref{Mm}.
To this end, points $x\in \R^n$ where $\alpha_V(x)=\{\ell(x)\}$ is a singleton are easily characterized because they satisfy $x\in \inn(C_{\ell(x)})$. Instead, consider any $x\in \R^n$, such that $\alpha_V(x)=\{\ell_1, \dots, \ell_p\}$ with $p>1$, namely any point $x$ where the locally Lipschitz function $V$ is not continuously differentiable. Define now the probability simplex of dimension $m$ as 
\[
\Lambda^m_0:=\{\lambda\in \R^m_{\geq0}\;\vert\;\sum_{j=1}^m\lambda_j=1\}.
\]
Denoting $I(x) = \{i_1, \dots, i_m\}\subseteq \{1, \dots, M\}$, and proceeding as in Remark~\ref{oss:construction}, by~\eqref{eq:Mmdirder} we have that $\dot{\overline{V}}_{F^\sw_{\text{lin}}}(x) \neq \emptyset$ if and only if there exist $\lambda=(\lambda_1,\dots,\lambda_m)\in \Lambda^m_0$ such that
\begin{equation}\label{eq:linearLiequality}
\nabla V_{\ell_{k+1}}(x)^\top \left( \sum_{j=1}^m \lambda_j A_{i_j} x \right)=\nabla V_{\ell_k}(x)^\top \left(\sum_{j=1}^m \lambda_i A_{i_j} x \right),
\end{equation}
for each $k \in \{1, \dots, p-1\}$. Based on~\eqref{eq:linearLiequality}, define the set $\Lambda(x,\{A_{i}\}_{i\in I(x)})\subset\Lambda^m_0$ as
\begin{equation}\label{eq:bigsystem1}
\lambda\in \Lambda(x,\{A_{i}\}_{i\in I(x)})\,\Leftrightarrow
\begin{cases}
\sum_{j=1}^m \lambda_j x^\top (P_{\ell_2}-P_{\ell_1})A_{i_j} x =0,\\
 \hskip 0.75cm \vdots \hskip0.75cm \vdots \hskip0.75cm \vdots \\
\sum_{j=1}^m \lambda_j x^\top (P_{\ell_p}-P_{\ell_{p-1}})A_{i_j} x =0,
\end{cases}
\end{equation}
where $\lambda=(\lambda_1,\dots,\lambda_m)\in \Lambda^m_0$.
Then, recalling~\eqref{eq:Mmdirder}, we have 
\begin{equation}\label{eq:doubleimplication}
\ell \in \alpha_V(x)\;\Rightarrow\;\dot{\overline{V}}_{F^\sw_{\text{lin}}}(x)=\left\{ \begin{aligned} & 2x^\top P_{\ell}(\lambda_1 A_{i_1}+\dots + \lambda_m A_{i_m})x  \; : \; \\ & (\lambda_1, \cdots, \lambda_m) \in \Lambda(x,\{A_{i}\}_{i\in I(x)}) \end{aligned} \right\}.
\end{equation}
The equivalence  \eqref{eq:doubleimplication} is used to prove the next corollary of Theorem~\ref{thm:switchstab}.
\begin{cor}\label{cor:linstab}
Consider system \eqref{eq:linear} and a max-min of quadratics $V \in \Mmq(P_1, \dots, P_K)$, where $P_1, \dots, P_K$ are symmetric, positive-definite,  and pairwise distinct matrices.
Suppose that there exists $\varepsilon > 0$ such that
\begin{enumerate}[leftmargin=0.6cm]
\item[(i)] For each $x \in \R^n$ with $\alpha_V(x)=\{\ell\}$ and $I(x)=\{i\}$ being singletons, it holds that
\begin{equation}\label{eq:FirstCondLinCor}
x^\top (A_i^\top P_{\ell}+P_{\ell}A_i)x \le - \varepsilon \vert x \vert^2.
\end{equation}
\item[(ii)] For each $x \in \R^n$ satisfying $\alpha_V(x)=\{\ell_1,\dots, \ell_p\}\subset \{1,\dots, K\}$, with $p>1$, and $I(x)=\{i_1,\dots, i_m\}\subset\{1,\dots,M\}$ with $m>1$, there exists $\ell \in \alpha_V(x)$ such that
\begin{equation}\label{eq:SecondCondLinCor}
\sum_{i\in I(x)} \lambda_i x^\top (P_{\ell} A_i + A_i^\top P_{\ell})x \le  - \varepsilon \vert x \vert^2,
\end{equation}
for all $(\lambda_1, \dots, \lambda_m) \in \Lambda(x,\{A_i\}_{i\in I(x)})$.
\end{enumerate}
Then the origin of \eqref{eq:linear} is GAS.
\end{cor}
\begin{proof}
It follows from Theorem~\ref{thm:switchstab} that the origin of \eqref{eq:linear} is GAS if \eqref{eq:prop1ineq} holds for all $x \in \R^n$. We will proceed by analyzing four cases, depending on whether the sets $I(x)$ and $\alpha_V(x)$ are singletons or not.\\ First, consider $x$ such that $\alpha_V(x)=\{\ell\}$ and $I(x)=\{i\}$ are singletons. In this case,
$$
\dot{\overline{V}}_{F^\sw_{\text{lin}}}(x) = \left \{x^\top (A_i^\top P_{\ell}+P_{\ell}A_i)x \right\} \le  - \varepsilon \vert x \vert^2,
$$
where the inequality is due to condition {\em (i)}.\\
Secondly, for a point $x$ with $\alpha_V(x)=\{\ell_1,\dots, \ell_p\}$, with $p>1$, and $I(x) = \{i_1,\dots, i_m\}$ with $m>1$, it follows from~\eqref{eq:doubleimplication} and condition \emph{(ii)} that $\max \dot{\overline{V}}_{F^\sw_{\text{lin}}} (x) \le - \varepsilon \vert x \vert^2$.\\
Next, consider the case where $\alpha_V(x)=\{\ell\}$ is a singleton and $I(x)=\{i_1,\dots, i_m\}$ with $m>1$, that is a point where $V$ is continuously differentiable and the set $F^\sw_{\text{lin}}(x)$ in~\eqref{eq:linear} is multivalued. We thus have $\partial V(x)=\{\nabla V_\ell(x)\}$, and from linearity we have
\begin{equation}\label{eq:LieDeriLinearCorollary}
\max \dot{\overline{V}}_{F^\sw_{\text{lin}}} (x)\leq \max_{\lambda\in \Lambda^m_0} \sum_{j=1}^m 2 \lambda_j x^\top P_\ell A_{i_j}x =2x^\top P_\ell A_{i^\star} x,
\end{equation}
where $i^\star\in\argmax_{i=i_1,\dots,i_m}2x^\top P_\ell A_{i} x$. Since $i^\star\in I(x)$, by~\eqref{eq:defJ} $x\in \overline{D_{i^\star}}$ ; from item~\emph{(i)} we have 
\[
x_k^\top (A_{i^\star}^\top P_{\ell}+P_{\ell}A_{i^\star})x_k \le - \varepsilon \vert x_k \vert^2. 
\]
for some sequence $x_k\to x$ with $x_k\in D_{i^\star}\cap\inn(C_\ell)$, $\forall k\in \N$.
By continuity we thus have $x^\top (A_{i^\star}^\top P_{\ell}+P_{\ell}A_{i^\star})x \le - \varepsilon \vert x \vert^2$, and from~\eqref{eq:LieDeriLinearCorollary} we have
$
\max \dot{\overline{V}}_{F^\sw_{\text{lin}}}(x)\le  - \varepsilon \vert x \vert^2
$.\\
Finally, we consider the case $\alpha_V(x)=\{\ell_1,\dots,\ell_p\}$ with $p>1$ and $I(x)=\{i\}$, namely a point where the function $V$ is not continuously differentiable and the set $F^\sw_{\text{lin}}(x)$ is a singleton, since $x\in D_i$.
If $\dot{\overline{V}}_{F^\sw_{\text{lin}}}(x)=\emptyset$ we are done. Otherwise, in view of~\eqref{eq:bigsystem1}, $\dot{\overline{V}}_{F^\sw_{\text{lin}}}(x)\neq\emptyset$ implies
\begin{equation}\label{eq:LieDeriLinearCorollary2}
\{2x^\top P_{\ell_1}A_ix\}=\dots=\{2x^\top P_{\ell_p}A_ix\}=\dot{\overline{V}}_{F^\sw_{\text{lin}}}(x).
\end{equation}
Considering, without loss of generality, the index $\ell_1\in \alpha_V(x)$, by Definition~\ref{def:alpha} and recalling that $D_i$ is open, we can consider a sequence $x_k\to x$ such that $x_k\in D_i\cap\inn(C_{\ell_1})$, for all $k\in \N$. By condition \emph{(i)} we have
\[
x_k^\top (A_i^\top P_{\ell_1}+P_{\ell_1}A_i)x_k \le - \varepsilon \vert x_k \vert^2, \;\;\forall\, k\in \N. 
\]
By continuity  $x^\top (A_i^\top P_{\ell_1}+P_{\ell_1}A_i)x \le - \varepsilon \vert x \vert^2$; recalling~\eqref{eq:LieDeriLinearCorollary2}, it implies that
$\max \dot{\overline{V}}_{F^\sw_{\text{lin}}}(x)\le  - \varepsilon \vert x \vert^2
$.\\
Having analyzed all the cases, we conclude that \eqref{eq:prop1ineq} holds for all $x \in \R^n$ and the assertion follows from Theorem~\ref{thm:switchstab}.
\end{proof}

\subsection{Checking Item (i) of Corollary~\ref{cor:linstab}}\label{subsec:Item(i)}
In this section, we exploit the properties of system~\eqref{eq:linear} and the family of candidate max-min Lyapunov functions in~\eqref{Mm} to computationally check condition \emph{(i)} of Corollary~\ref{cor:linstab}. We do so by following two steps: first, fixing $K\geq 1$, $J\geq1$, nonempty subsets $S_1,\dots, S_J\subset\{1,\dots, K\}$, and hence the corresponding max-min combination in~\eqref{Mm}, we construct an auxiliary function $\Phi$, which characterizes the regions where $\alpha_V:\R^n\rightrightarrows\{1,\dots,K\}$ is single-valued. Notably, this function is independent of $P_1,\dots,P_K$.
Secondly, we use $\Phi$ to compute matrices $P_1,\dots, P_K$ satisfying item~\emph{(i)} of Corollary~\ref{cor:linstab} by only checking the feasibility of a finite set of matrix inequalities.
The details of implementing these two steps now follow:
\setcounter{algocf}{-1}
\begin{algorithm}[!b]\caption{The function $\Phi:\bS_K\to \{1,\dots K\}$.}\label{alg:FunctionPhi}
\KwData{$K\in \N$, $J\geq 1$, $S_1,\dots,S_J\subset\{1,\dots K\}$}
\KwIn{$\rho=(\rho_1,\dots, \rho_K)\in\bS_K$}
\KwOut{$\text{out}=\Phi_{K,J,S_1,\dots,S_J}(\rho)$}
\SetKwFunction{FMain}{$\Phi_{K,J,S_1,\dots,S_J}$}
    \SetKwProg{Fn}{Function}{:}{}
    \Fn{\FMain{$\rho$}}{
        Set: $\text{out}=0$, $S_{\min}=\emptyset$,    

        \For{$ (j=1,\, j\leq J,\,j=j+1) $}{
        \For{$(i= 1,\,i\leq K,\,i=i+1)$}{ \If{$\rho_i\in S_j$}
        {Add $\rho_i$ to $S_{\min}$, \textbf{break} }  } 
    
} \For{$(j=J,\, j\geq0, \,j=j-1)$}{\If{$\rho_j\in S_{\min}$} {$\text{out}=\rho_j$, \textbf{break}} }

\textbf{return} $\text{out} $}
\vskip0.2cm
\textbf{End Function}
\end{algorithm}

\setcounter{step}{-1}
\begin{step}\label{rmk:SelectionMap}
Consider  the symmetric group of order $K$ denoted by $\bS_K$, which is the group of all possible permutations of the first $K$ positive integers.
Given \emph{any} $K$  pairwise distinct quadratic functions associated to some $P_1, \dots, P_K>0$, for any $\rho=(\rho_1, \dots , \rho_K) \in \bS_K$, define the open set
\begin{equation}\label{eq:cone}
E_{\rho} := \left\{ x \in \mathbb{R}^n \; \vert \; x^\top P_{\rho_1}x < \dots <  x^\top P_{\rho_K}x\right\},
\end{equation}
which is a cone (possibly empty) where a strict ordering among the $K$ quadratic functions holds.
For a given max-min combination in~\eqref{Mm}, namely given  $J\geq 1$ and nonempty sets $S_j \subset \{1, \dots, K\}$, $\forall j \in \{1, \dots, J \}$, in each $E_\rho$ the function $\alpha_V:\R^n\rightrightarrows\{1,\dots,K\}$ defined in~\eqref{eq:alpha} is constant and single valued; let us denote it by $\Phi(\rho):=\alpha_V(E_\rho)\in \{1,\dots ,K\}$.
 \hfill$\triangle$
\end{step}

In Algorithm~\ref{alg:FunctionPhi}, we present how to numerically construct $\Phi:\bS_K\to\{1,\dots,K\}$, independently of matrices $(P_1,\dots,P_K)$.

\begin{oss}
We emphasize that function $\Phi$ is independent of $P_1,\dots, P_K$, but only depends on the max-min policy defined by sets $S_1,\dots S_J$. As an example, considering $J=K$ and $S_j=\{j\}$, the max-min combination~\eqref{Mm} coincides with the maximum of the $K$ quadratic functions. In this case, $\Phi$ will be defined as
$
\Phi((\rho_1,\dots,\rho_K))=\rho_K,$ $\forall \rho=(\rho_1,\dots,\rho_K)\in \bS_K,
$
because of~\eqref{eq:cone}. Also, to relate $\Phi$ with $\alpha_V$, it is seen that for any $K$ base quadratics defined by $(P_1,\dots,P_K)$ with a specific max-min combination determined by $V$, the mapping $\alpha_V$ in~\eqref{eq:alpha} corresponds to 
$
\alpha_V(x)=\bigcap_{\varepsilon>0}\{\Phi(\rho)\,\vert\,E_\rho\cap \B(x,\varepsilon)\neq \emptyset\}.
$
\end{oss}
Next, in Step~\ref{rmk:S-PROC}, we use the function $\Phi$ to check condition \emph{(i)} of Corollary~\ref{cor:linstab}:
\begin{step}[Conditions on $E_\rho$]\label{rmk:S-PROC}
Consider system~\eqref{eq:linear}, and take $K\in \N$, $J\geq 1$  and $S_1,\dots S_J\subset\{1,\dots K\}$ nonempty sets. 
Find $P_1,\dots,P_K>0$, $\beta_i(\rho)\geq0$, $\tau_{i,k}(\rho)\geq0$,\;$\forall\,\rho=(\rho_1,\dots,\rho_K)\in \bS_K$, $\forall\, k\in \{1,\dots ,K-1\}$,  and $\forall i\in \{1,\dots ,M\}$,  such that
\begin{equation}\label{eq:BMiCondition}
\begin{aligned}
A_i^\top P_{\Phi(\rho)}\hskip-0.05cm+P_{\Phi(\rho)}A_i+\hskip-0.05cm \sum_{k=1}^{K-1}\tau_{i,k}(\rho)(P_{\rho_{k+1}}-P_{\rho_k})+\beta_i(\rho) Q_i<0.
\end{aligned}
\end{equation}
\end{step}
In Proposition~\ref{prop:AlgoStep1} below, we prove that the feasibility of Step~\ref{rmk:S-PROC} yields $K$ matrices such that condition \emph{(i)} of Corollary~\ref{cor:linstab} holds, while in Algorithm~\ref{alg:LyapunovAlg} we formalize this step of computationally checking  condition~\eqref{eq:BMiCondition}.

\begin{algorithm}[!t]\caption{Lyapunov conditions: Differentiable case.}\label{alg:LyapunovAlg}
\small
\KwData{$A_1,\dots, A_M\in \R^{n\times n}$, $Q_1,\dots,Q_M\in \Sym(\R^{n})$.}
\textsc{Initialization: }Choose the max-min structure:\\
{Take $K\in \N$, $J\geq 1$, $S_1,\dots S_J\subset\{1,\dots K\}$, construct $\Phi:\bS_K\to \{1,\dots,K\}$ (Algorithm~\ref{alg:FunctionPhi}).
}\\
\textsc{Lyapunov conditions on  $E_\rho$, $\forall \rho\in \bS_K$:}\\
{ (Step~\ref{rmk:S-PROC}){ Check the feasibility of
{\footnotesize
\begin{equation}\label{eq:Feasability}
\begin{aligned}
\hskip-0.4cm &A_i^\top P_{\Phi(\rho)}+P_{\Phi(\rho)}A_i+\sum_{k=1}^{K-1}\tau_{i,k}(\rho)(P_{\rho_{k+1}}-P_{\rho_k}) +\beta_i(\rho)Q_i<0,\\
&P_1,\dots P_K>0,
  \beta_i(\rho),\tau_{i,k}(\rho)\geq0,\;\forall\, \rho=(\rho_1,\dots,\rho_K)\in \bS_K,\\ 
&k\in \{1,\dots, K-1\}, i\in \{1,\dots, M\}.
\end{aligned}
\end{equation}} }\\\eIf{ \eqref{eq:Feasability} are feasible}{\KwOut{Matrices $(P_1,\dots, P_K)$} }{\KwOut{$\emptyset$}}}
\end{algorithm}

\begin{prop}\label{prop:AlgoStep1}
Consider $K\in \N$, $J\geq 1$, $S_1,\dots,S_J\subset\{1,\dots K\}$ non-empty, $(P_1,\dots, P_K)$ positive definite matrices and $V$ defined as in~\eqref{Mm}. If, for any $\rho=(\rho_1,\dots,\rho_K)\in \bS_K$, any $i\in \{1,\dots, M\}$ and any $k=\{1,\dots,K-1\}$, there exist $\beta_i(\rho)\geq0$, $\tau_{i,k}(\rho)\geq0$ such that \eqref{eq:BMiCondition} holds, then item~\emph{(i)} of Corollary~\ref{cor:linstab} holds.
\end{prop}
\begin{proof}
The set $E_\rho\cap D_i$ can be written as
\[
E_\rho\cap D_i=\left\{
x\in \R^n\bigg\vert\;\begin{aligned} &x^\top Q_ix>0\wedge \,x^\top(P_{\rho_{k+1}}-P_{\rho_k})x>0,\\ &\forall \;k\in \{1,\dots,K-1\}   \;
\end{aligned}\right\}.
\]
If~\eqref{eq:BMiCondition} holds, due to the strict inequality, there exists $\varepsilon_{i,\rho}>0$ such that
\begin{equation}\label{eq:IntermediateStepBMi}
x^\top (A_i^\top P_{\Phi(\rho)}+P_{\Phi(\rho)}A_i)x\leq -\varepsilon_{i,\rho}|x|^2,\;\;\forall\,x\in D_i\cap E_\rho.
\end{equation}
 By Step~\ref{rmk:SelectionMap} we have $\alpha_V(x)=\{\Phi(\rho)\}$ and, by~\eqref{eq:SwitchingCones}, $I(x)=\{i\}$ for all $D_i\cap E_\rho$, and thus~\eqref{eq:IntermediateStepBMi} implies that~\eqref{eq:FirstCondLinCor} holds 
for all $x\in D_i\cap E_\rho$. Defining  $\varepsilon:=\min_{i,\rho}\varepsilon_{i,\rho}$ we have that~\eqref{eq:FirstCondLinCor} holds for each $x\in \R^n$  with $\alpha_V(x)$ and $I(x)$ being singletons, thus concluding the proof.
\end{proof}
\begin{oss}[Polyhedral cones]
Consider again the alternative state-space partition discussed in Remark~\ref{rmk:StatePartition}. More precisely, consider polyhedral cones $D_1,\dots, D_M\subset \R^n$ defined by $D_i:=\{x\in \R^n\;\vert\; K_ix \geq_c \mathbf{0}\}$, 
where, for each $i\in \{1,\dots, M\}$ $K_i\in \R^{k_i\times n}$, for some $k_i\in \N$. Equivalently, the sets $D_i$ can be represented by $D_i=\text{cone}(v_i)_{i=1}^{M_i}$, where $v_1,\dots,v_{M_i}\in \R^n$ are the rays of the cone $D_i$. Let us call by $R_i\in \R^{n\times M_i}$ the matrix whose columns are the vectors $v_i$. As presented in~\cite[Lemma~1]{IerTan17} we have
that, given any symmetric matrix $S\in \R^{n\times n}$, if there exists a symmetric and entry-wise positive matrix $N_i\in \R^{n \times n}$ such that $R_i^\top S R_i+N_i\leq 0$ 
then $x^\top S x<0,\;\forall x\in D_i$.
Using this result, the procedure presented in Step~\ref{rmk:S-PROC} and Proposition~\ref{prop:AlgoStep1} can be adapted to the polyhedral cones case by requiring that for any $\rho=(\rho_1,\dots,\rho_K)\in \bS_K$, any $i\in \{1,\dots, M\}$ and any $k=\{1,\dots,K-1\}$, there exist $\tau_{i,k}(\rho)\geq0$ and a symmetric entry-wise positive matrix $N_i(\rho)$  such that 
\[
R_i^\top S(\rho) R_i+N_i(\rho)\leq 0, 
\]
with $S(\rho):=A_i^\top P_{\Phi(\rho)}\hskip-0.05cm+P_{\Phi(\rho)}A_i+ \sum_{k=1}^{K-1}\tau_{i,k}(\rho)(P_{\rho_{k+1}}-P_{\rho_k})$. 
\end{oss}
\begin{oss}[Computational burden]\label{rmk:ReducingINeq}
It is noted that, in general, since $|\bS_K|= K!$, Algorithm~\ref{alg:LyapunovAlg} requires studying the feasibility of $M\cdot K!$ inequalities, which involve $M K K!$ non-negative scalars and $K$ symmetric positive-definite matrices. It is clear that the computational burden grows quickly as a function of the number $K$ of the chosen base-quadratics. However, fixing $J\geq 1$, $S_1,\dots,S_J\subset\{1,\dots, K\}$ in \eqref{Mm} (thus fixing a particular max-min structure) the computational burden can be reduced. In \cite[]{dellarossa18} we showed how the number of required inequalities depends on the choice of sets $S_j$ in the case of three quadratics, i.e. $K=3$.
\end{oss}

\begin{example}[\ref{shieldexample}{ - Continued: Item~\emph{(i)}}]
We have already proved that there does not exist a convex Lyapunov function for system (\ref{eq:exShieldSys}). 
We will construct a max-min of quadratics Lyapunov function $V$ of the form~\eqref{eq:exShieldLyap}.
In other words, we have fixed  $K=3$, $J=2$, $S_1=\{1,2\}$ and $S_2=\{3\}$. Using Algorithm~\ref{alg:FunctionPhi} we construct the function $\Phi$ that reads
$
\Phi(\rho_1)=\Phi(\rho_2)=\Phi(\rho_3)=\Phi(\rho_4) = 3, 
$ 
 where $\rho_1=(1,2,3)$, $\rho_2=(1,3,2)$, $\rho_3=(2,1,3)$, $\rho_4=(2,3,1)$; and
$
\Phi(\rho_5)=1,
$ 
 where $\rho_5=(3,1,2)$; and
$
\Phi(\rho_6)=2
$, where $\rho_6=(3,2,1)$.
In these cases, the matrix inequalities of Algorithm~\ref{alg:LyapunovAlg} (after the reductions outlined in Remark~\ref{rmk:ReducingINeq}) read
\begin{align*}
&A_2^\top P_2+P_2 A_2+\tau_1 (P_2-P_3)+\tau_2 (P_1-P_2)+\beta_1Q_2<0,\\
&A_1^\top P_1+P_1 A_1+\tau_3 (P_1-P_3)+\tau_4 (P_2-P_1)+\beta_2Q_1<0,\\
&A_3^\top P_3+P_3 A_3+\tau_5(P_3-P_1)+\beta_3Q_3<0,\\
&A_3^\top P_3+P_3 A_3+\tau_6 (P_3-P_2)+\tau_7(P_1-P_3)+\beta_4 Q_4<0,\\
&\tau_k\geq0,\,\forall k\in \{1,\dots, 7\},\, \beta_i\geq0,\;\forall i\in \{1,\dots, 4\},\; P_1,P_2,P_3>0.
\end{align*}
Using numerical solvers, it follows that these inequalities are feasible, and in particular they are satisfied by
\begin{equation}\label{eq:ShieldExampleMatricesP}
P_1= \begin{bmatrix}
    5 & 0 \\
    0 & 1
  \end{bmatrix}, \quad
  P_2= \begin{bmatrix}
    1 & 0 \\
    0 & 5
  \end{bmatrix}, \quad
  P_3= \begin{bmatrix}
   3 & 2 \\
    2 & 3
    \end{bmatrix},
\end{equation}
$\tau=(0.258,0.102,0.258, 0.102, 0.284, 0.193, 0.090)$ and $\beta_i=0$, $\forall\,i\in \{1,\dots, 4\}$. A level set of $V$ is plotted in Fig.~\ref{flower2plot}. This proves that $V$ in~\eqref{eq:exShieldLyap} with $P_i$ as in~\eqref{eq:ShieldExampleMatricesP} satisfies item~\emph{(i)} of Corollary~\ref{cor:linstab}.
\end{example}
\subsection{Checking item (ii) of Corollary~\ref{cor:linstab} in $\R^2$.}
To study GAS of system~\eqref{eq:linear}, we also need to check item \emph{(ii)} of Corollary~\ref{cor:linstab}, which is computationally harder than item \emph{(i)}. We now discuss how this condition simplifies in the planar case, that is when $n=2$.
To do so, let us analyze the geometry of the switching rule proposed in~\eqref{eq:SwitchingCones}. To non-trivially satisfy Assumption~\ref{ass:swSig}, we will suppose that the matrices $Q_1,\dots, Q_M\in \Sym(\R^2)$ are sign indefinite. We will characterize the sets $D_i$ in~\eqref{eq:SwitchingCones} using the following result.
\begin{lemma}\label{lemma:PlanarDecomposition}
 Given any sign indefinite matrix $Q\in \Sym(\R^2)$, there exist $\theta_1,\theta_2\in \R^2\setminus\{0\}$, $\theta_2\notin \mathrm{span}(\theta_1)$ such that 
 \begin{equation}\label{eq:PlanaDecomposition}
Q=\theta_1\theta_2^\top+\theta_2\theta_1^\top.
 \end{equation}
\end{lemma}
\begin{proof}[Sketch of the proof]
 Let us denote by $\lambda_-<0<\lambda_+$ the eigenvalues of $Q$, and with $v_-$,$v_+\in \R^2$ the corresponding unit eigenvectors  ($|v_-|=|v_+|=1$). By the spectral decomposition we  have that
 $
Q=\lambda_+ v_+v_+^\top+\lambda_- v_-v_-^\top.
 $
 Let us call $\eta=\sqrt{\frac{-\lambda_-}{\lambda_+-\lambda_-}}>0$ and $\kappa=\sqrt{\frac{2}{\lambda_+-\lambda_-}}>0$, then by choosing
 $
\theta_1=\kappa\left[\sqrt{1-\eta^2}\;v_+-\eta\, v_-\right ]$ and 
$\theta_2=\kappa\left[\sqrt{1-\eta^2}\;v_++\eta\, v_-\right ]$,
 it is seen that~\eqref{eq:PlanaDecomposition} holds.
\end{proof}
Lemma~\ref{lemma:PlanarDecomposition} allows checking algorithmically condition~\emph{(ii)} of Corollary~\ref{cor:linstab} in the planar case. This is done in two steps.
\renewcommand{\substyle}[1]{\alph{#1}}
\sublabon{step}
\begin{step}\label{step:ParametrizingQi}
Given $M$ sign indefinite matrices $Q_1,\dots, Q_M\in \Sym(\R^2)$ that satisfy Assumption~\ref{ass:swSig}, the non-overlapping and covering conditions in~Assumption~\ref{ass:swSig} imply that matrices $Q_i$, $i=1,..., M$, decomposed as in~\eqref{eq:PlanaDecomposition}, can be suitably ordered\footnote{For the ordering of matrices $Q_i$, via vectors $\theta_i$ in~\eqref{eq:PlanaDecomposition}, we can associate an angle with each one of the lines $\theta_i$, $i = 1,\dots,M$, using the \emph{atan2} function.}
in such a way that 
\begin{equation}\label{eq:LinesVsPartition}
\begin{aligned}
Q_i&=\theta_i\theta_{i+1}^\top+\theta_{i+1}\theta_i^\top\;\; \text{for } i=1,\dots, M-1,\\
Q_M&=\theta_M(-\theta_{1})^\top+(-\theta_{1})\theta_M^\top,
\end{aligned}
\end{equation}
for some suitable selections of linear independent vectors $\theta_1,\dots,\theta_M\in \R^2\cap\{(x_1,x_2)\in \R^2\,\vert\, x_1\geq0\}$.\\
For each $i\in \{1,\dots, M\}$, take $v_i\in \R^2$ as an unit vector generating the subspace $\theta_i^\perp:=\{x\in \R^2\;\vert\; \theta_i^\top x=0\}$.\hfill$\triangle$
\end{step}
\begin{step}\label{step:CheckingOnLines}
Consider $V\in \Mmq(P_1 , \dots, P_K )$ satisfying condition~(i) of Corollary~\ref{cor:linstab}. For every $v_i$ such that $\alpha_V(v_i)=\{\ell^i_1,\ell^i_2\}$ is multivalued, solve the system
\begin{equation}\label{eq:LinearSystemPlanar}
\begin{cases}
0\leq\lambda\leq 1,\\
 \lambda v_i^\top (P_{\ell^i_2}-P_{\ell^i_1})A_{i-1} v_i+(1-\lambda)v_i^\top (P_{\ell^i_2}-P_{\ell^i_1})A_{i} v_i= 0,\\
\end{cases}
\end{equation}
(with $i-1=M$ if $i=1$), and denote by $\Lambda^i\subset[0,1]$ the set of solutions of~\eqref{eq:LinearSystemPlanar} for $v_i$ (possibly empty).\hfill$\triangle$
\end{step}

\sublaboff{step}
In the following we formally prove the effectiveness of Steps~\ref{step:ParametrizingQi} and~\ref{step:CheckingOnLines}.

\begin{prop}\label{prop:PlanarStability}
Consider $A_1,\dots ,A_M\in \R^{2\times2}$ and $M$ indefinite matrices $Q_1,\dots Q_M\in \Sym(\R^2)$ that satisfy Assumption~\ref{ass:swSig}, and are parameterized as in \eqref{eq:LinesVsPartition}. Suppose that there exist $P_1,\dots, P_K > 0$ such that $V \in \Mmq(P_1, \dots, P_K)$ satisfies condition~{(i)} of Corollary~\ref{cor:linstab}.  If, for all $i\in \{1,\dots,M\}$ such that $\alpha_V(v_i)$ is multivalued, we have
\begin{multline}\label{eq:PlanarSecondCondition}
\lambda v_i^\top (P_{\ell^i_1} A_{i-1} + A_{i-1}^\top P_{\ell^i_1})v_i +(1-\lambda) v_i^\top (P_{\ell^i_1} A_i + A_i^\top P_{\ell^i_1})v_i<0,
\end{multline}
for all $\lambda\in \Lambda^i$, then item (ii) of Corollary~\ref{cor:linstab} holds.
\end{prop}
\begin{proof}
Recalling~\eqref{eq:SwitchingCones}, the parametrization in~\eqref{eq:LinesVsPartition}  characterizes the points $x$ where the map $I(x)$ is multivalued. From~\eqref{eq:LinesVsPartition}, we have that $\overline{D}_i\cap\overline{D}_{i+1}=\theta_{i+1}^\perp$, for all $i=1,\dots,M-1$ and $\overline{D}_M\cap\overline{D}_{1}=\theta_{1}^\perp$.
Thus
\begin{equation}\label{eq:PlanarSystemI}
I(x)=\begin{cases}\{i,i+1\},\;\;&\text{if } x\in \theta_{i+1}^\perp, \;i=1,\dots, M-1,\\
\{1,M\},\;\;&\text{if } x\in \theta_{1}^\perp,\\
\{i\}, \;\;&\text{if } x^\top Q_i x > 0.
\end{cases}
\end{equation}
Let us now consider a function $V\in \Mmq(P_1, \dots, P_K)$ that satisfies condition~\emph{(i)} of Corollary~\ref{cor:linstab}.
From Remark~\ref{rmk:Homogeneity}, for any max-min function $V \in \Mmq(P_1, \dots, P_K)$, the value of the map $\alpha_V:\R^2\rightrightarrows\{1,\dots,K\}$ has \emph{at most} $2$ elements.
To check item~\emph{(ii)} of Corollary~\ref{cor:linstab} we must consider all the points $x\in \R^2$ such that $I(x)$ and $\alpha_V(x)$ are multivalued. As shown in~\eqref{eq:PlanarSystemI}, the set of points where the map $I$ is multivalued coincides with the union of the $M$ lines $\theta^\perp_1,\dots, \theta^\perp_M$.
From Remark~\ref{rmk:Homogeneity}, the homogeneity of $F^\sw_{\text{lin}}$ and $V\in\Mmq(P_1, \dots, P_K)$ implies that it is sufficient to check condition \emph{(ii)} of Corollary~\ref{cor:linstab} only for the chosen unit vectors $v_1,\dots v_M$ which span $\theta^\perp_1,\dots \theta^\perp_M$ respectively. We can conclude noting that, for each $i\in \{1,\dots,M\}$ such that $\alpha_V(v_i)$ is multivalued, system~\eqref{eq:LinearSystemPlanar} corresponds to~\eqref{eq:bigsystem1}, and equation~\eqref{eq:SecondCondLinCor} follows from~\eqref{eq:PlanarSecondCondition} selecting a small enough $\varepsilon>0$.
\end{proof}

Proposition~\ref{prop:PlanarStability} shows that for a \emph{planar} linear switched system~\eqref{eq:linear},~\eqref{eq:SwitchingCones} involving $M$ subsystems, it is sufficient to identify unit vectors $v_i$, $i = 1, \dots, M$ generating the switching lines, and verify inequality~\eqref{eq:SecondCondLinCor} for these $M$ points. Item \emph{(ii)} of Corollary~\ref{cor:linstab} then follows from homogeneity.
This result allows concluding the analysis of Example~\ref{shieldexample}.
\begin{example}[\ref{shieldexample} {- Continued: Item~\emph{(ii)}}]
As a last step to show that the origin of~\eqref{eq:exShieldSys},~\eqref{eq:signal} is GAS, we have to ensure the condition \emph{(ii)} of Corollary~\ref{cor:linstab}. Since the signal~\eqref{eq:signal} can be rewritten in the form~\eqref{eq:LinesVsPartition}, we can follow Steps~\ref{step:ParametrizingQi} and~\ref{step:CheckingOnLines}, taking $v_1\in\cS_{13}$, $v_2\in \cS_{21}$, $v_3\in \cS_{32}$ such that $|v_j|=1$, for all $j\in \{1,2,3\}$. Considering system~\eqref{eq:LinearSystemPlanar}, it is easily checked that
$
\Lambda^j=\emptyset,\;\;\forall\;j\in \{1,2,3\}
$.
Recalling~\eqref{eq:doubleimplication}, $\dot{\overline{V}}_F({v_j})=\emptyset$, for $j=1,2,3$.
 Then by Proposition~\ref{prop:PlanarStability} the function $V$ in \eqref{eq:exShieldLyap} is a Lyapunov function for system \eqref{eq:exShieldSys} which certifies GAS.
\end{example}

\begin{oss}\label{rem:ExClarkFails}
In Example~\ref{shieldexample}, it can be shown that $V$ in \eqref{eq:exShieldLyap} does not satisfy the conditions $\dot V_{F^\sw}(x) < 0$ for some $x \in \R^2$: consider the point $v_1 \in \cS_{13}$, where we have shown  $\dot{\overline{V}}_{F^\sw}(v_1)=\emptyset$.
Since $v_1\in \cS_{13}$, then
$
\partial V(v_1) =\co\{ 2P_1 v_1, 2P_3v_1\}$ and $
F^\sw(z_0) =\co\{A_1 v_1, A_3 v_1\}$.
Straightforward computations yield
$
v_1^\top (P_3 A_1+A_1^\top P_3) v_1=8.65 >0,
$
and thus $\exists w \in \partial V(v_1)$ and $f\in F^\sw(v_1)$ such that
$
0 < w^\top f \in \dot V_{F^\sw} (v_1),
$
which implies that Corollary \ref{mainteo} is not applicable and well illustrate the fact that Corollary~\ref{cor:linstab} provides less conservative conditions.
\end{oss}

\subsection{Checking item~{(ii)} of Corollary~\ref{cor:linstab} in $\R^n$ with 2 modes.}
The main difficulty in checking item~{(ii)} of Corollary~\ref{cor:linstab} in higher dimensions is that the set $\Lambda(x, \{A_i\}_{i \in I(x)})$ of \eqref{eq:bigsystem1} cannot be finitely parameterized.
In this section, we impose a structure on \eqref{eq:linear} which allows us to check this condition without explicitly computing $\Lambda(x, \{A_i\}_{i \in I(x)})$. The idea is to rule out the motion on switching surfaces, in which case negative definiteness of $\dot{\overline{V}}_{F^\sw_{\text{lin}}}$ on the switching surface can be established by continuity arguments.
More precisely, we consider a 2-mode $n$-dimensional switched system, i.e. $M=2$ in~\eqref{eq:linear}, and the sets $D_1,D_2\subset\R^n$ are defined by
\begin{equation}\label{eq:2ModePartition}
\begin{aligned}
D_1&=\{x\in \R^n\;\vert\;x^\top Q_1x=x^\top Qx>0\},\\
D_2&=\{x\in \R^n\;\vert\;x^\top Q_2x=-x^\top Qx>0\},
\end{aligned}
\end{equation}
where $Q\in \Sym(\R^{n})$ is invertible.
In other words, we are considering a partition of $\R^n$ as in Assumption~\ref{ass:swSig}, which comprises two symmetric cones $D_1,D_2\subset \R^n$. We will denote the boundary of these cones (also called the \emph{switching surface}) with $\cQ:=\{x\in \R^n\;\vert\;x^\top Q x=0\}$. 
A computationally attractive way to avoid sliding motion, is to follow a preliminary step, presented in what follows.
 \begin{figure*}[t!]
        \centering
        \includegraphics[height=3.3cm]{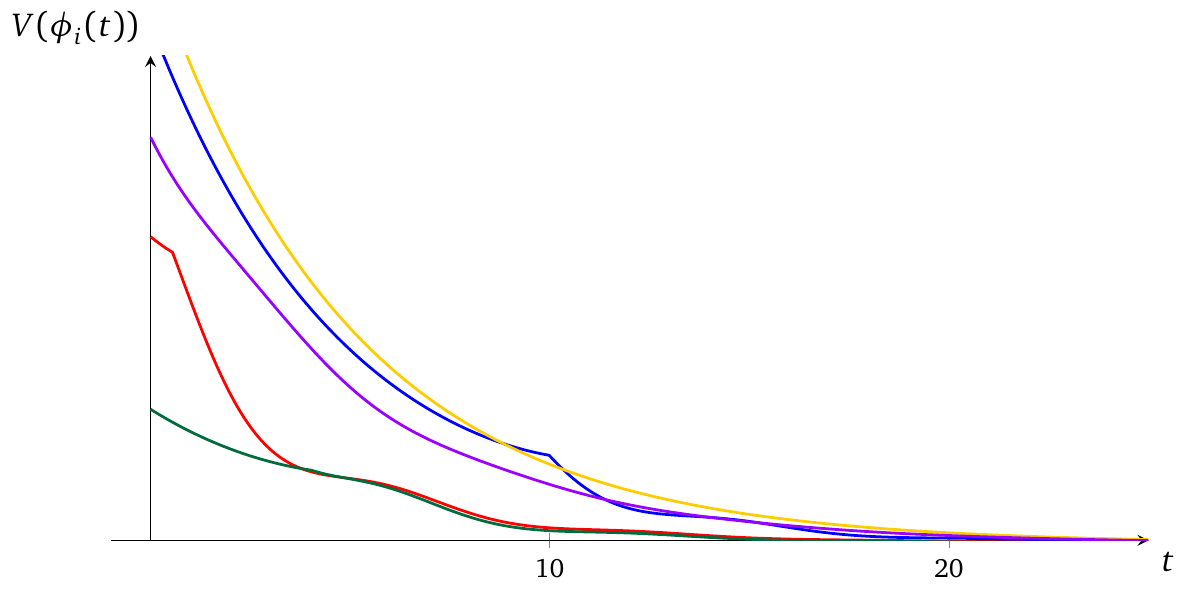}
        \caption{The evolution of the Lyapunov function $V$ along  solutions $\phi_i$, $i=1,\dots,5$. }
    \label{figure:3d}
\end{figure*}
\begin{stepbis}{step:ParametrizingQi}[Ruling out motion on the switching surface]\label{step:RulingSliding}
For every $z\in \R^n$ such that $|z|=1$, check if the implication
\begin{equation}\label{eq:SigmaMatrix}
z^\top Qz=0 \;\Rightarrow\; (z^\top QA_1z)(z^\top QA_2z)>0
\end{equation}
is satisfied.\hfill$\triangle$
\end{stepbis}
Condition~\eqref{eq:SigmaMatrix} intuitively means that, given a unit vector $z\in\cQ$, the vectors $A_1z$ and $A_2z$ are both pointing inside (or outside) the cone $D_1$, and thus it rules out the possibility of having solutions sliding along $\cQ$.
A viable way to check condition \eqref{eq:SigmaMatrix} is to  consider the decomposition of $Q$ as $Q=V\Lambda V^\top$, where invertibility of $Q$ implies that $\Lambda$ is a diagonal matrix with only $1$ and $-1$  diagonal elements and then check the simpler implication 
\[
\oz^\top \Lambda \oz=0\;\Rightarrow (\oz^\top V^{-\top}QA_1V^{-1}\oz) (\oz^\top V^{-\top}QA_2V^{-1}\oz)>0
\]
for all $|\oz|=1$.
To simplify the discussion, consider max-min combination over $K$ quadratics defined by $K$ symmetric and positive definite matrices $P_1, \dots P_K \in \mathbb{R}^{n \times n}$ satisfying:
\begin{equation}\label{eq:NonDegenerate}
\rank(P_{j_1}-P_{j_2})=n,\;\;\forall j_1,j_2\in\{1,\dots, K\},j_1\neq j_2,
\end{equation}
which is not too restrictive since full-rank matrices are dense in $\R^{n\times n}$.
\begin{prop}\label{Prop:StabilityMultiDim}
Consider a 2-mode linear switched system~\eqref{eq:linear} with $D_1,D_2\subset\R^n$ as in~\eqref{eq:2ModePartition} and $Q\in \Sym(\R^{n})$ invertible. Suppose that for all $z\in \R^n$, $|z|=1$, the implication~\eqref{eq:SigmaMatrix} in Step~\ref{step:RulingSliding} holds. If there exist $P_1,\dots,P_K > 0 $ satisfying~\eqref{eq:NonDegenerate}, and $V \in \Mmq(P_1, \dots, P_K)$ satisfying condition~{(i)} of Corollary~\ref{cor:linstab}, then item (ii) holds and system~\eqref{eq:linear} is GAS.
\end{prop}
\begin{proof}
To check item~\emph{(ii)} of Corollary~\ref{cor:linstab}, consider any $x\in \R^n$ such that $x^\top Q x=0$, i.e. $I(x)=\{1,2\}$, and $\alpha_V(x)=\{\ell_1,\dots, \ell_p\}\subset \{1,\dots, K\}$ with $p>1$. We consider 2 cases:\\
\emph{Case 1: Suppose there exist $\ell',\ell''\in \alpha_V(x)$, $\ell'\neq \ell''$ such that $x^\top (P_{\ell'}-P_{\ell''})=\tau x^\top Q$, for some $\tau\in\R\setminus \{0\}$.}\\
Then the equation (resembling~\eqref{eq:linearLiequality},) 
\[
x^\top P_{\ell'}(\lambda A_1 x+(1-\lambda)A_2x)=x^\top P_{\ell''}(\lambda A_1 x+(1-\lambda)A_2x),
\]
has solutions $\lambda\in[0,1]$ if and only if there exists $\lambda\in [0,1]$ such that
\begin{equation}\label{eq:NoSwLemma:Equality1}
x^\top Q(\lambda A_1x +(1-\lambda)A_2x)=0.
\end{equation}
We have supposed that~\eqref{eq:SigmaMatrix} in Step~\ref{step:RulingSliding} holds for $z=\frac{x}{|x|}$, thus by homogeneity of $F^\sw_{\text{lin}}(x)$ equation~\eqref{eq:NoSwLemma:Equality1} has no solution $\lambda\in[0,1]$ since the scalars
$x^\top Q A_1 x$ and $x^\top Q A_2 x$ have the same sign (and are not zero). Recalling equations~\eqref{eq:bigsystem1} and~\eqref{eq:doubleimplication}, this implies that $\dot{\overline{V}}_{F^\sw_{\text{lin}}}(x)=\emptyset$, ensuring \emph{(ii)} of Corollary~\ref{cor:linstab}.\\
\emph{Case 2: Suppose that $\forall\; \ell',\ell''\in \alpha_V(x)$, $\ell'\neq \ell''$,  $x^\top (P_{\ell'}-P_{\ell''})\neq \tau x^\top Q_1$, for all $\tau\in\R\setminus \{0\}$.}\\
In this case we show in Lemma~\ref{lemma:TechincaLemma} in~\ref{sec:AppendixTecLemma} that there exists a sequence $x_k\to x$ such that $x_k\in \cQ$, (i.e. $I(x_k)=\{1,2\}$) and $\alpha_V(x_k)=\{\oell\}$, for all $k\in \N$, for an $\oell\in \alpha_V(x)$. By hypothesis, $V$ satisfies item~\emph{(i)} of Corollary~\ref{cor:linstab}, implying by continuity that, for every $k \in \N$,
 \[
 x_k^\top P_{\oell} A_1x_k\leq-\varepsilon|x_k|^2, \text{ and } x_k^\top P_{\oell} A_2x_k\leq-\varepsilon|x_k|^2.
 \]
Since $x_k\to x$ when $k\to \infty$, again by continuity we have
 \[
x^\top P_{\oell} A_1x\leq-\varepsilon|x|^2 \;\text{ and }\;x^\top P_{\oell} A_2x\leq-\varepsilon|x|^2.
 \]
 Thus, $\forall\lambda\in [0,1]$ such that
 \[
x^\top P_{\ell_1}(\lambda A_1+(1-\lambda)A_2)x=\dots=x^\top P_{\ell_p}(\lambda A_1x+(1-\lambda)A_2x),
 \]
 we have $x^\top P_{\oell}(\lambda A_1+(1-\lambda)A_2)x\leq-\varepsilon|x|^2$. Recalling ~\eqref{eq:bigsystem1} and~\eqref{eq:doubleimplication} it implies that $\max\dot{\overline{V}}_{F^\sw_{\text{lin}}}(x)\leq -\varepsilon|x|^2$.\\
 Having proved item~\emph{(ii)} of Corollary~\ref{cor:linstab} in both \emph{Cases 1} and \emph{2}, we can conclude.
 \end{proof}

 \begin{example}\label{ex:TriDExample}
Concluding this section, we present a switched system evolving in $\R^3$ and we prove GAS using Proposition~\ref{Prop:StabilityMultiDim}. \\ Let us consider the matrices,
\begin{equation}\label{eq:exampleSystemdata}
A_1=\begin{bsmallmatrix}
-0.1 & -1& 0\\1 & -0.1 & 0\\ 0 & 0& 0.2
\end{bsmallmatrix},\;A_2=\begin{bsmallmatrix}
-0.2 & 1& 0.1\\-1 & -0.2 & 0\\ 0.1 & 0& -0.1
\end{bsmallmatrix},\;
Q=\begin{bsmallmatrix}
1 & 0 & 0\\ 0 & 1 & 0\\ 0 & 0 & -1
\end{bsmallmatrix}
\end{equation}
and $Q_1=Q,\;\;Q_2=-Q$.
It is easy to see that they define a system of the form~\eqref{eq:linear} and moreover $Q$ is invertible. Parameterizing a generic $x\in \cQ=\{x\in \R^3\;\vert\;x^\top Q x=0\}$ as 
$
x=\begin{bsmallmatrix}
x_1,~x_2,~ \pm \sqrt{x_1^2+x_2^2}
\end{bsmallmatrix}^\top,
$
it can  be seen that~\eqref{eq:SigmaMatrix} in Step~\ref{step:RulingSliding} holds. Using the algorithms~\ref{alg:FunctionPhi} and~\ref{alg:LyapunovAlg} of Section~\ref{subsec:Item(i)}, we prove here that the max of 2 quadratics defined by
\[
V(x):=\max\{x^\top P_1x ,x^\top P_2x\}
\]
with $P_1:=\begin{bsmallmatrix}  4 & 0& 0\\ 0& 4 & 0\\0&0&1\end{bsmallmatrix}$ and $P_2:=\begin{bsmallmatrix}  3 & 0& 0\\ 0& 3 & 0\\0&0&2\end{bsmallmatrix}$ satisfies item~\emph{(i)} of Corollary~\ref{cor:linstab}.
First of all, we have that $P_1-P_2=Q$, and thus the analysis outlined in Step~\ref{rmk:SelectionMap} is simplified, since
\[
\begin{aligned}
&E_{\rho_1}=\{x\in \R^3\vert\;x^\top Q_1 x>0\}=:D_1\;\text{and}\;\Phi(\rho_1)=1,\\
&E_{\rho_2}=\{x\in \R^3\vert\;x^\top Q_2 x>0\}=:D_2\;\text{and}\;\Phi(\rho_2)=2,
\end{aligned}
\]
where $\rho_1=(1,2)$ and $\rho_2=(2,1)$ denote the two elements of $\bS_2$.
Following Step~\ref{rmk:S-PROC}, item \emph{(i)} of Corollary~\ref{cor:linstab} holds, since $P_1A_1+A_1^\top P_1+\tau_1 Q_1<0$ and $P_2A_2+A_2^\top P_2+\tau_2 Q_2<0$
are satisfied choosing $\tau_1=0.6, \tau_2=0$. Since \eqref{eq:SigmaMatrix}~and~\eqref{eq:NonDegenerate} hold, invoking Proposition~\ref{Prop:StabilityMultiDim} we have that item~\emph{(ii)} of Corollary~\ref{cor:linstab} is satisfied, and $V$ is a Lyapunov function proving GAS of system~\eqref{eq:exampleSystemdata}.
In Figure~\ref{figure:3d}, we have plotted the evolution of $V$ along 5 particular solutions of system~\eqref{eq:exampleSystemdata}.

 \end{example}

\section{Conclusions}\label{sec:conc}
For the class of systems comprising differential inclusions, and state-dependent switched systems, we introduced a family of nonsmooth functions obtained by max-min combinations. Based on two notions of generalized directional derivatives, we proposed sufficient conditions for global asymptotic stability. For a class of systems with conic switching regions and linear dynamics within each of these regions, we studied some conditions under which a max-min condition can be obtained by solving matrix inequalities. A possible route for future research is the generalization of this approach to a wider class of systems, and develop further numerical tools for checking the proposed Lie derivative based conditions.


\appendix
\section{A Technical Lemma}\label{sec:AppendixTecLemma}
\begin{lemma}\label{lemma:TechincaLemma}
Consider $Q\in \Sym(\R^{n})$ invertible and  any max-min function $V \in \Mm(x^\top P_1 x, \dots, x^\top P_Kx)$, such that $P_1,\dots,P_K>0$ satisfy~\eqref{eq:NonDegenerate}. Consider a point $x\in \R^n$ such that $x^\top Qx=0$ and $\alpha_V(x)=\{\ell_1,\dots,\ell_p\}$ ($p>1$). If  $\forall\; \ell',\ell''\in \alpha_V(x)$, $\ell'\neq \ell''$,  $x^\top (P_{\ell'}-P_{\ell''})\neq \tau x^\top Q$, for all $\tau\in\R\setminus \{0\}$, then there exists a sequence $x_k\to x$ such that $x_k^\top Q x_k=0$ and  $\alpha_V(x_k)=\{\oell\}$, for all $k\in \N$, for an $\oell\in \alpha_V(x)$.
\end{lemma}
\begin{proof}
Since $Q$ is invertible, for all $x\in \R^n\setminus \{0\}$ such that $x\in \cQ=\{x\in \R^n\;\vert\;x^\top Q x=0\}$, we can define the tangent space of $\cQ$ at $x$ as $T_x(\cQ):=\{w\in \R^n\;\vert\;x^\top Q w=0\}$, see for example~\cite[Page 23]{barden2003introduction}.
Consider $v \in \R^n$, $v \neq 0$, such that $x^\top Q v=0$ and $x^\top (P_{\ell'}-P_{\ell''})v\neq 0$, for all $\ell',\ell''\in \alpha_V(x)$, $\ell'\neq \ell''$. Such a $v\in \R^n$ exists,  since, by~\eqref{eq:NonDegenerate},  $Q$ and $P_{\ell'}-P_{\ell''}$ are invertible and $x^\top Q$ and $x^\top (P_{\ell'}-P_{\ell''})$ are linearly independent, for all $\ell',\ell''\in \alpha_V(x)$, $\ell'\neq \ell''$. By definition of  $T_x(\cQ)$, given $\beta>0$, there exists  continuously differentiable function $\psi:(-\beta,\beta)\to\R^n$ such that $\psi(0)=x$, $\dot \psi(0)=v$, and $\psi(\tau)^\top Q \psi(\tau)=0$, (i.e. $ \psi(\tau)\in\cQ$), $\forall\;\tau\in (-\beta,\beta)$.
For all $\ell',\ell''\in \alpha_V(x)$, $\ell'\neq \ell''$, define $\Psi_{\ell',\ell''}:(-\beta,\beta)\to\R$ as
\[
\Psi_{\ell',\ell''}(\tau):=\psi(\tau)^\top (P_{\ell'}-P_{\ell''})\psi(\tau).
\]
 Since $\ell', \ell''\in \alpha_V(x)$ we have $\Psi_{\ell',\ell''}(0)=0$; moreover $\Psi_{\ell',\ell''}$ is continuously differentiable at $0$ and by the chain rule $\dot \Psi_{\ell',\ell"}(0)=x^\top (P_{\ell'}-P_{\ell''})v\neq 0$. This means that there exists a $\beta'< \beta$, $\beta'\neq 0$ such that
 \begin{equation}\label{eq:permanenceofSign}
 \Psi_{\ell',\ell''}(\tau)=\psi(\tau)^\top P_{\ell'}\psi(\tau)-\psi(\tau)^\top P_{\ell''}\psi(\tau)\neq 0,
 \end{equation}
 for all $\tau\in (0,\beta')$, for all $\ell',\ell''\in \alpha_V(x)$, $\ell'\neq \ell''$.
 We now consider a sequence $\tau_k\to 0$ such that $\tau_k\in (0,\beta')$, $\forall k\in \N$, and define $x_k:=\psi(\tau_k)$. Without loss of generality we can suppose $x_k\in \cU$, for all $k\in \N$, where $\cU$ is the open neighborhood of $x$ defined in Lemma~\ref{lemma:prelimnLemma}. By~\eqref{eq:permanenceofSign} we have
 \[
x_k^\top P_{\ell'} x_k\neq x_k^\top P_{\ell''} x_k, \;\;\forall\,\ell',\ell''\in \alpha_V(x),\,\ell'\neq \ell''.
 \]
 By Lemma~\ref{lemma:prelimnLemma} this implies that,  $\forall \, k\in \N$, there exists $\ell_k\in\alpha_V(x)$ such that $\alpha_V(x_k)=\{\ell_k\}$. By finiteness of $\alpha_V(x)$, possibly considering a subsequence, we can suppose $\alpha_V(x_k)=\{\oell\}$ $\forall\,k\in \N$, with $\oell\in \alpha_V(x)$. Since, by definition of $\psi$,  $x_k^\top Qx_k=0$, $\forall\,k\in \N$, we also have $I(x_k)=\{1,2\}$, $\forall k\in \N$. 
\end{proof}

\footnotesize
\bibliography{biblio}

\end{document}